\documentclass[12pt]{amsart}

\usepackage{color}
\usepackage{amsmath, amsthm, amssymb}
\usepackage{amsfonts}
\usepackage[ansinew]{inputenc}
\usepackage[dvips]{epsfig}
\usepackage{graphicx}
\usepackage[english]{babel}
\usepackage{hyperref}
\theoremstyle{plain}
\newtheorem{thm}{Theorem}[section]
\newtheorem{cor}[thm]{Corollary}
\newtheorem{lem}[thm]{Lemma}
\newtheorem{prop}[thm]{Proposition}
\newtheorem{claim}[thm]{Claim}

\theoremstyle{definition}
\newtheorem{defi}[thm]{Definition}

\theoremstyle{remark}
\newtheorem{rem}[thm]{Remark}

\numberwithin{equation}{section}

\newcommand{\average}{{\mathchoice {\kern1ex\vcenter{\hrule height.4pt
width 6pt depth0pt} \kern-9.7pt} {\kern1ex\vcenter{\hrule
height.4pt width 4.3pt depth0pt} \kern-7pt} {} {} }}

\def\R{\mathbb{R}}

\begin{document}

\title[The Pohozaev identity for the fractional Laplacian]{The Pohozaev identity for the fractional Laplacian}

\author{Xavier Ros-Oton}

\address{Universitat Polit\`ecnica de Catalunya, Departament de Matem\`{a}tica  Aplicada I, Diagonal 647, 08028 Barcelona, Spain}
\email{xavier.ros.oton@upc.edu}

\thanks{The authors were supported by grants MTM2008-06349-C03-01, MTM2011-27739-C04-01 (Spain), and 2009SGR345 (Catalunya)}

\author{Joaquim Serra}

\address{Universitat Polit\`ecnica de Catalunya, Departament de Matem\`{a}tica  Aplicada I, Diagonal 647, 08028 Barcelona, Spain}

\email{joaquim.serra@upc.edu}


\maketitle

\begin{abstract} In this paper we prove the Pohozaev identity for the semilinear Dirichlet
problem
$(-\Delta)^s u =f(u)$ in $\Omega$,
$u\equiv0$ in $\mathbb R^n\backslash\Omega$.
Here, $s\in(0,1)$, $(-\Delta)^s$ is the fractional Laplacian in $\R^n$, and $\Omega$ is a bounded $C^{1,1}$ domain.

To establish the identity we use, among other things, that if $u$ is a bounded solution then $u/\delta^s|_{\Omega}$ is $C^{\alpha}$ up to the boundary $\partial \Omega$, where $\delta(x)={\rm dist}(x,\partial\Omega)$.
In the fractional Pohozaev identity, the function $u/\delta^s|_{\partial\Omega}$ plays the role that $\partial u/\partial\nu$ plays in the classical one.
Surprisingly, from a nonlocal problem we obtain an identity with a boundary term (an integral over $\partial\Omega$) which is completely local.

As an application of our identity, we deduce the nonexistence of nontrivial solutions in star-shaped domains for supercritical nonlinearities.
\end{abstract}


\section{Introduction and results}

Let $s\in(0,1)$ and consider the fractional elliptic problem
\begin{equation}\label{eq}
\left\{ \begin{array}{rcll} (-\Delta)^s u &=&f(u)&\textrm{in }\Omega \\
u&=&0&\textrm{in }\mathbb R^n\backslash\Omega\end{array}\right.
\end{equation}
in a bounded domain $\Omega\subset\mathbb R^n$, where
\begin{equation}
\label{laps}(-\Delta)^s u (x)= c_{n,s}{\rm PV}\int_{\R^n}\frac{u(x)-u(y)}{|x-y|^{n+2s}}dy
\end{equation}
is the fractional Laplacian. Here, $c_{n,s}$ is a normalization constant given by \eqref{cns}.

When $s=1$, a celebrated result of S. I. Pohozaev states that any solution of \eqref{eq} satisfies an identity, which is known as the Pohozaev identity \cite{P}. This classical result has many consequences, the most immediate one being the nonexistence of nontrivial bounded solutions to \eqref{eq} for supercritical nonlinearities $f$.

The aim of this paper is to give the fractional version of this identity, that is, to prove the Pohozaev identity for problem \eqref{eq} with $s\in(0,1)$.
This is the main result of the paper, and it reads as follows.
Here, since the solution $u$ is bounded, the notions of weak and viscosity solutions agree (see Remark \ref{boundedsol}).

\begin{thm}\label{thpoh} Let $\Omega$ be a bounded and $C^{1,1}$ domain, $f$ be a locally Lipschitz function, $u$ be a bounded solution of (\ref{eq}), and
\[\delta(x)={\rm dist}(x,\partial\Omega).\]
Then,
\[u/\delta^s|_{\Omega}\in C^{\alpha}(\overline\Omega)\qquad \textrm{for some }\ \alpha\in(0,1),\]
meaning that $u/\delta^s|_{\Omega}$ has a continuous extension to $\overline\Omega$ which is
$C^{\alpha}(\overline{\Omega})$, and the following identity holds
\[(2s-n)\int_{\Omega}uf(u)dx+2n\int_\Omega F(u)dx=\Gamma(1+s)^2\int_{\partial\Omega}\left(\frac{u}{\delta^s}\right)^2(x\cdot\nu)d\sigma,\]
where $F(t)=\int_0^tf$, $\nu$ is the unit outward normal to $\partial\Omega$ at $x$, and $\Gamma$ is the Gamma function.
\end{thm}

Note that in the fractional case the function $u/\delta^s|_{\partial\Omega}$ plays the role that $\partial u/\partial\nu$ plays in the classical Pohozaev identity.
Moreover, if one sets $s=1$ in the above identity one recovers the classical one, since $u/\delta|_{\partial\Omega}=\partial u/\partial\nu$ and $\Gamma(2)=1$.

It is quite surprising that from a nonlocal problem \eqref{eq} we obtain a completely local boundary term in the Pohozaev identity. That is, although the function $u$ has to be defined in all $\R^n$ in order to compute its fractional Laplacian at a given point, knowing $u$ only in a neighborhood of the boundary we can already compute $\int_{\partial\Omega}\left(\frac{u}{\delta^s}\right)^2(x\cdot\nu)d\sigma$.

Recall that problem \eqref{eq} has an equivalent formulation given by the Caffarelli-Silvestre \cite{CSext} associated extension problem ---a local PDE in $\R^{n+1}_+$. For such extension,
some Pohozaev type identities are proved in \cite{BCP,CC,CT}.
However, these identities contain boundary terms on the cylinder $\partial\Omega\times \mathbb R^+$ or in a half-sphere $\partial B_R^+\cap\mathbb R^{n+1}_+$, which have no clear interpretation in terms of the original problem in $\R^n$.
The proofs of these identities are similar to the one of the classical Pohozaev identity and use PDE tools (differential calculus identities and integration by parts).

Sometimes it may be useful to write the Pohozaev identity as
\[2s[u]^2_{H^s(\mathbb R^n)}-2n\mathcal{E}[u]=\Gamma(1+s)^2\int_{\partial\Omega}\left(\frac{u}{\delta^s}\right)^2(x\cdot\nu)d\sigma,\]
where $\mathcal{E}$ is the energy functional
\begin{equation}\label{energy}
\mathcal{E}[u]=\frac12[u]^2_{H^s(\mathbb R^n)}-\int_\Omega F(u)dx,
\end{equation}
$F'=f$, and
\begin{equation}\label{seminorm}
[u]_{H^s(\mathbb R^n)}=\||\xi|^s\mathcal F[u]\|_{L^2(\mathbb R^n)}=\frac{c_{n,s}}{2}\int_{\mathbb R^n}\int_{\mathbb R^n}\frac{\left|u(x)-u(y)\right|^2}{|x-y|^{n+2s}}dxdy.
\end{equation}
We have used that if $u$ and $v$ are $H^s(\R^n)$ functions and $u\equiv v\equiv0$ in $\R^n\setminus\Omega$, then
\begin{equation}\label{fip}
\int_\Omega v(-\Delta)^su\, dx=\int_{\R^n}(-\Delta)^{s/2}u(-\Delta)^{s/2}v\, dx,\end{equation}
which yields
\[\int_\Omega uf(u)dx=\int_{\mathbb R^n}|(-\Delta)^{s/2}u|^2dx=[u]_{H^s(\mathbb R^n)}.\]

As a consequence of our Pohozaev identity we obtain nonexistence results for problem \eqref{eq} with supercritical nonlinearities $f$ in star-shaped domains $\Omega$.
In Section \ref{sec2} we will give, however, a short proof of this result using our method to establish the Pohozaev identity.
This shorter proof will not require the full strength of the identity.

\begin{cor}\label{cornonexistence} Let $\Omega$ be a bounded, $C^{1,1}$, and star-shaped domain, and let $f$ be a locally Lipschitz function.
If
\begin{equation}\label{supercritic} \frac{n-2s}{2n}uf(u)\geq\int_0^u f(t)dt\qquad\textrm{for all}\ \ u\in\mathbb R, \end{equation}
then problem \eqref{eq} admits no positive bounded solution. Moreover, if the inequality in \eqref{supercritic} is strict, then \eqref{eq} admits no nontrivial bounded solution.
\end{cor}

For the pure power nonlinearity, the result reads as follows.

\begin{cor}\label{cornonexistence2} Let $\Omega$ be a bounded, $C^{1,1}$, and star-shaped domain. If $p\geq\frac{n+2s}{n-2s}$, then problem
\begin{equation}\label{eqp}\left\{ \begin{array}{rcll} (-\Delta)^s u &=&|u|^{p-1}u&\textrm{in }\Omega \\
u&=&0&\textrm{in }\mathbb R^n\backslash\Omega\end{array}\right.\end{equation}
admits no positive bounded solution. Moreover, if $p>\frac{n+2s}{n-2s}$ then \eqref{eqp} admits no nontrivial bounded solution.
\end{cor}

The nonexistence of changing-sign solutions to problem \eqref{eqp} for the critical power $p=\frac{n+2s}{n-2s}$ remains open.

Recently, M. M. Fall and T. Weth \cite{FW} have also proved a nonexistence result for problem \eqref{eq} with the method of moving spheres.
In their result no regularity of the domain is required, but they need to assume the solutions to be positive.
Our nonexistence result is the first one allowing changing-sign solutions.
In addition, their condition on $f$ for the nonexistence ---\eqref{condFW} in our Remark \ref{remFW}--- is more restrictive than ours, i.e., \eqref{supercritic} and,
when $f=f(x,u)$, condition \eqref{supercritic2}.

The existence of weak solutions $u\in H^s(\R^n)$ to problem \eqref{eq} for subcritical $f$ has been recently proved by R. Servadei and E. Valdinoci \cite{SV}.

The Pohozaev identity will be a consequence of the following two results.
The first one establishes $C^s(\R^n)$ regularity for $u$, $C^{\alpha}(\overline\Omega)$ regularity for $u/\delta^s|_\Omega$, and higher order interior H\"older estimates for $u$ and $u/\delta^s$. It is proved in our paper \cite{RS}.

Throughout the article, and when no confusion is possible, we will use the notation $C^\beta(U)$ with $\beta>0$ to refer to the space $C^{k,\beta'}(U)$, where $k$ is the is greatest integer such that $k<\beta$, and $\beta'=\beta-k$. This notation is specially appropriate when we work with $(-\Delta)^s$ in order to avoid the splitting of different cases in the statements of regularity results.
According to this, $[\cdot]_{C^{\beta}(U)}$ denotes the $C^{k,\beta'}(U)$ seminorm
\[[u]_{C^\beta(U)}=[u]_{C^{k,\beta'}(U)}=\sup_{x,y\in U,\ x\neq y}\frac{|D^ku(x)-D^ku(y)|}{|x-y|^{\beta'}}.\]
Here, by $f\in C^{0,1}_{\rm loc}(\overline\Omega\times\R)$ we mean that $f$ is Lipschitz in every compact subset of $\overline\Omega\times \R$.

\begin{thm}[\cite{RS}]\label{krylov} Let $\Omega$ be a bounded and $C^{1,1}$ domain, $f\in C^{0,1}_{\rm loc}(\overline\Omega\times\R)$, $u$ be a bounded solution of
\begin{equation}\label{eqlin}
\left\{ \begin{array}{rcll} (-\Delta)^s u &=&f(x,u)&\textrm{in }\Omega \\
u&=&0&\textrm{in }\mathbb R^n\backslash\Omega,\end{array}\right.\end{equation}
and $\delta(x)={\rm dist}(x,\partial\Omega)$. Then,
\begin{itemize}
\item[(a)] $u\in C^s(\R^n)$ and, for every $\beta\in[s,1+2s)$, $u$ is of class $C^{\beta}(\Omega)$ and
\[[u]_{C^{\beta}(\{x\in\Omega\,:\,\delta(x)\ge\rho\})}\le C \rho^{s-\beta}\qquad \textrm{for all}\ \  \rho\in(0,1).\]
\item[(b)] The function $u/\delta^s|_\Omega$ can be continuously extended to $\overline\Omega$.
Moreover, $u/\delta^s$ belongs to $C^{\alpha}(\overline{\Omega})$ for some $\alpha\in(0,1)$ depending only on $\Omega$, $s$, $f$, $\|u\|_{L^{\infty}(\R^n)}$.
In addition, for all $\beta\in[\alpha,s+\alpha]$, it holds the estimate
\[ [u/\delta^s]_{C^{\beta}(\{x\in\Omega\,:\,\delta(x)\ge\rho\})}\le C \rho^{\alpha-\beta}\qquad \textrm{for all}\ \  \rho\in(0,1).\]
\end{itemize}
The constant $C$ depends only on $\Omega$, $s$, $f$, $\|u\|_{L^{\infty}(\R^n)}$, and $\beta$.
\end{thm}

\begin{rem}\label{boundedsol}
For bounded solutions of \eqref{eqlin}, the notions of energy and viscosity solutions coincide (see more details in Remark 2.9 in \cite{RS}).
Recall that $u$ is an energy (or weak) solution of problem \eqref{eqlin} if $u\in H^s(\mathbb R^n)$, $u\equiv0$ in $\mathbb R^n\backslash\Omega$, and
\[\int_{\mathbb R^n}(-\Delta)^{s/2}u(-\Delta)^{s/2}v\,dx=\int_\Omega f(x,u)v\,dx\]
for all $v\in H^s(\R^n)$ such that $v\equiv0$ in $\R^n\setminus\Omega$.

By Theorem \ref{krylov} (a), any bounded weak solution is continuous up to the boundary and solve equation \eqref{eqlin} in the classical sense, i.e., in the pointwise sense of \eqref{laps}.
Therefore, it follows from the definition of viscosity solution (see \cite{CS}) that bounded weak solutions are also viscosity solutions.

Reciprocally, by uniqueness of viscosity solutions \cite{CS} and existence of weak solution for the linear problem $(-\Delta)^sv=f(x,u(x))$, any viscosity solution $u$ belongs to $H^s(\R^n)$ and it is also a weak solution. See \cite{RS} for more details.
\end{rem}

The second result towards Theorem \ref{thpoh} is the new Pohozaev identity for the fractional Laplacian. The hypotheses of the following proposition are satisfied for any bounded solution $u$ of \eqref{eqlin} whenever $f\in C^{0,1}_{\rm loc}(\overline\Omega\times\R)$, by our results in \cite{RS} (see Theorem \ref{krylov} above).

\begin{prop}\label{intparts} Let $\Omega$ be a bounded and $C^{1,1}$ domain. Assume that $u$ is a $H^s(\R^n)$ function which vanishes in $\R^n\setminus\Omega$, and satisfies
\begin{itemize}
\item[(a)] $u\in C^s(\R^n)$ and, for every $\beta\in[s,1+2s)$, $u$ is of class $C^{\beta}(\Omega)$ and
\[[u]_{C^{\beta}(\{x\in\Omega\,:\,\delta(x)\ge\rho\})}\le C \rho^{s-\beta}\qquad \textrm{for all}\ \  \rho\in(0,1).\]
\item[(b)] The function $u/\delta^s|_\Omega$ can be continuously extended to $\overline\Omega$. Moreover, there exists $\alpha\in(0,1)$ such that $u/\delta^s\in C^{\alpha}(\overline{\Omega})$. In addition, for all $\beta\in[\alpha,s+\alpha]$, it holds the estimate
\[ [u/\delta^s]_{C^{\beta}(\{x\in\Omega\,:\,\delta(x)\ge\rho\})}\le C \rho^{\alpha-\beta}\qquad \textrm{for all}\ \  \rho\in(0,1).\]
\item[(c)] $(-\Delta)^s u$ is pointwise bounded in $\Omega$.
\end{itemize}
Then, the following identity holds
\[\int_\Omega(x\cdot\nabla u)(-\Delta)^su\ dx=\frac{2s-n}{2}\int_{\Omega}u(-\Delta)^su\ dx-\frac{\Gamma(1+s)^2}{2}\int_{\partial\Omega}\left(\frac{u}{\delta^s}\right)^2(x\cdot\nu)d\sigma,\]
where $\nu$ is the unit outward normal to $\partial\Omega$ at $x$, and $\Gamma$ is the Gamma function.
\end{prop}

\begin{rem} Note that hypothesis (a) ensures that $(-\Delta)^s u$ is defined pointwise in $\Omega$.
Note also that hypotheses (a) and (c) ensure that the integrals appearing in the above identity are finite.
\end{rem}

\begin{rem}
By Propositions 1.1 and 1.4 in \cite{RS}, hypothesis (c) guarantees that $u\in C^s (\R^n)$ and $u/\delta^s\in C^\alpha(\overline\Omega)$, but not the interior estimates in (a) and (b).
However, under the stronger assumption $(-\Delta)^s u\in C^\alpha(\overline\Omega)$ the whole hypothesis (b) is satisfied; see Theorem 1.5 in \cite{RS}.
\end{rem}

As a consequence of Proposition \ref{intparts}, we will obtain the Pohozaev identity (Theorem \ref{thpoh}) and also a new integration by parts formula related to the fractional Laplacian.
This integration by parts formula follows from using Proposition \ref{intparts} with two different origins.

\begin{thm}\label{corintparts} Let $\Omega$ be a bounded and $C^{1,1}$ domain, and $u$ and $v$ be functions satisfying the hypotheses in Proposition \ref{intparts}.
Then, the following identity holds\footnote{In the previous version of this paper the sign of the boundary term in the following identity was incorrect. We thank Gerd Grubb for pointing this to us.}
\[\int_\Omega (-\Delta)^su\ v_{x_i}\,dx=-\int_\Omega u_{x_i}(-\Delta)^sv\,dx-\Gamma(1+s)^2\int_{\partial\Omega}\frac{u}{\delta^{s}}\frac{v}{\delta^{s}}\,\nu_i\,d\sigma\]
for $i=1,...,n$, where $\nu$ is the unit outward normal to $\partial\Omega$ at $x$, and $\Gamma$ is the Gamma function.
\end{thm}

To prove Proposition \ref{intparts} we first assume the domain $\Omega$ to be star-shaped with respect to the origin.
The result for general domains will follow from the star-shaped case, as seen in Section \ref{sec8}.
When the domain is star-shaped, the idea of the proof is the following.
First, one writes the left hand side of the identity as
\[\int_\Omega(x\cdot\nabla u)(-\Delta)^su\ dx=\left.\frac{d}{d\lambda}\right|_{\lambda=1^+}\int_{\Omega} u_\lambda (-\Delta)^s u\ dx,\]
where
\[u_\lambda(x)=u(\lambda x).\]
Note that $u_\lambda\equiv0$ in $\mathbb R^n\backslash \Omega$, since $\Omega$ is star-shaped and we take $\lambda>1$ in the above derivative.
As a consequence, we may use \eqref{fip} with $v=u_\lambda$ and make the change of variables $y=\sqrt{\lambda}x$, to obtain
\[\int_{\Omega}u_\lambda (-\Delta)^s u\ dx=\int_{\R^n} (-\Delta)^{s/2}u_\lambda (-\Delta)^{s/2} u\ dx=\lambda^{\frac{2s-n}{2}}\int_{\mathbb R^n} w_{\sqrt{\lambda}} w_{1/\sqrt{\lambda}}\ dy,\]
where
\[w(x)=(-\Delta)^{s/2}u(x).\]
Thus,
\begin{equation}\label{1207}\begin{split}
\int_\Omega(x\cdot\nabla u)(-\Delta)^su\ dx&= \left.\frac{d}{d\lambda}\right|_{\lambda=1^+}\left\{\lambda^{\frac{2s-n}{2}}\int_{\mathbb R^n} w_{\sqrt{\lambda}} w_{1/\sqrt{\lambda}}\ dy\right\}\\
&=\frac{2s-n}{2}\int_{\R^n}w^2dx+\left.\frac{d}{d\lambda}\right|_{\lambda=1^+}I_{\sqrt\lambda}\\
&=\frac{2s-n}{2}\int_{\R^n}u(-\Delta)^s u\,dx+\frac12\left.\frac{d}{d\lambda}\right|_{\lambda=1^+}I_{\lambda}
,\end{split}\end{equation}
where
\[I_\lambda=\int_{\mathbb R^n} w_{{\lambda}} w_{1/{\lambda}}dy.\]
Therefore, Proposition \ref{intparts} is equivalent to the following equality
\begin{equation}\label{partdificil}-\left.\frac{d}{d\lambda}\right|_{\lambda=1^+}\int_{\mathbb R^n} w_{{\lambda}} w_{1/{\lambda}}\ dy=
\Gamma(1+s)^2\int_{\partial\Omega}\left(\frac{u}{\delta^s}\right)^2(x\cdot\nu)d\sigma.\end{equation}

The quantity $\frac{d}{d\lambda}|_{\lambda=1^+}\int_{\R^n}w_{\lambda}w_{1/\lambda}$ vanishes for any $C^1(\R^n)$ function $w$, as can be seen by differentiating under the integral sign.
Instead, we will prove that the function $w=(-\Delta)^{s/2} u$ has a singularity along $\partial\Omega$, and that \eqref{partdificil} holds.

Next we give an easy argument to give a direct proof of the nonexistence result for supercritical nonlinearities without using neither equality \eqref{partdificil} nor the behavior of $(-\Delta)^{s/2}u$; the detailed proof is given in Section \ref{sec2}.

Indeed, in contrast with the delicate equality \eqref{partdificil}, the inequality
\begin{equation}\label{num}
\left.\frac{d}{d\lambda}\right|_{\lambda=1^+}I_\lambda\leq0
\end{equation}
follows easily from Cauchy-Schwarz.
Namely,
\[I_\lambda\leq \|w_\lambda\|_{L^2(\mathbb R^n)}\|w_{1/\lambda}\|_{L^2(\mathbb R^n)}=\|w\|_{L^2(\mathbb R^n)}^2=I_1,\]
and hence \eqref{num} follows.

With this simple argument, \eqref{1207} leads to
\[-\int_\Omega (x\cdot\nabla u)(-\Delta)^s u\ dx\geq\frac{n-2s}{2}\int_\Omega u(-\Delta)^su\ dx,\]
which is exactly the inequality used to prove the nonexistence result of Corollary \ref{cornonexistence} for supercritical nonlinearities.
Here, one also uses that, when $u$ is a solution of \eqref{eq}, then
\[\int_\Omega (x\cdot\nabla u)(-\Delta)^s u\ dx=\int_\Omega(x\cdot\nabla u)f(u)dx=\int_{\Omega}x\cdot\nabla F(u)dx=-n\int_\Omega F(u)dx.\]

This argument can be also used to obtain nonexistence results (under some decay assumptions) for weak solutions of \eqref{eq} in the whole $\R^n$; see Remark \ref{remR^n}.

The identity \eqref{partdificil} is the difficult part of the proof of Proposition \ref{intparts}.
To prove it, it will be crucial to know the precise behavior of $(-\Delta)^{s/2}u$ near $\partial\Omega$ ---from both inside and outside $\Omega$.
This is given by the following result.

\begin{prop}\label{thlaps/2}
Let $\Omega$ be a bounded and $C^{1,1}$ domain, and $u$ be a function such that $u\equiv0$ in $\mathbb R^n\backslash\Omega$ and that $u$ satisfies {\rm (b)} in Proposition \ref{intparts}.
Then, there exists a $C^{\alpha}(\mathbb R^n)$ extension $v$ of $u/\delta^s|_{\Omega}$ such that
\begin{equation}\label{**}
(-\Delta)^{s/2}u(x)=c_1\left\{\log^-\delta(x)+c_2\chi_{\Omega}(x)\right\}v(x)+h(x)\quad \textrm{in } \ \R^n,
\end{equation}
where $h$ is a $C^{\alpha}(\R^n)$ function, $\log^-t=\min\{\log t,0\}$,
\begin{equation}\label{1208} c_1=\frac{\Gamma(1+s)\sin\left(\frac{\pi s}{2}\right)}{\pi},\qquad{\rm and}\qquad c_2=\frac{\pi}{\tan\left(\frac{\pi s}{2}\right)}.\end{equation}

Moreover, if $u$ also satisfies {\rm (a)} in Proposition \ref{intparts}, then for all $\beta\in(0,1+s)$
\begin{equation}\label{numeret}
[(-\Delta)^{s/2}u]_{C^{\beta}(\{x\in\R^n:\,\delta(x)\ge\rho\})}\leq C\rho^{-\beta}\qquad \textrm{for all}\ \  \rho\in(0,1),
\end{equation}
for some constant $C$ which does not depend on $\rho$.
\end{prop}

The values \eqref{1208} of the constants $c_1$ and $c_2$ in \eqref{**} arise in the expression for the $s/2$ fractional Laplacian, $(-\Delta)^{s/2}$, of the 1D function $(x_n^+)^s$,
and they are computed in the Appendix.

Writing the first integral in \eqref{partdificil} using spherical coordinates, equality \eqref{partdificil} reduces to a computation in dimension 1, stated in the following proposition. This result will be used with the function $\varphi$ in its statement being essentially the restriction of $(-\Delta)^{s/2} u$ to any ray through the origin. The constant $\gamma$ will be chosen to be any value in $(0,s)$.

\begin{prop}\label{propoperador} Let $A$ and $B$ be real numbers, and
\[\varphi(t)=A\log^-|t-1|+B\chi_{[0,1]}(t)+h(t),\]
where $\log^-t=\min\{\log t,0\}$ and $h$ is a function satisfying, for some constants $\alpha$ and $\gamma$ in $(0,1)$, and $C_0>0$, the following conditions:
\begin{itemize}
\item[(i)] $\|h\|_{C^{\alpha}([0,\infty))}\leq C_0$.
\item[(ii)] For all $\beta\in[\gamma,1+\gamma]$
\[\|h\|_{C^{\beta}((0,1-\rho)\cup(1+\rho,2))}\leq C_0 \rho^{-\beta}\qquad \textrm{for all}\ \  \rho\in(0,1).\]
\item[(iii)] $|h'(t)|\leq C_0 t^{-2-\gamma}$ and $|h''(t)|\leq C_0 t^{-3-\gamma}$ for all $t>2$.
\end{itemize}
Then,
\[-\left.\frac{d}{d\lambda}\right|_{\lambda=1^+}\int_{0}^{\infty} \varphi\left(\lambda t\right)\varphi\left(\frac{t}{\lambda}\right)dt=A^2\pi^2+B^2.\]

Moreover,  the limit defining this derivative is uniform among functions $\varphi$ satisfying (i)-(ii)-(iii) with given constants $C_0$, $\alpha$, and $\gamma$.
\end{prop}

From this proposition one obtains that the constant in the right hand side of \eqref{partdificil}, $\Gamma(1+s)^2$, is given by $c_1^2(\pi^2+c_2^2)$.
The constant $c_2$ comes from an involved expression and it is nontrivial to compute (see Proposition \ref{prop:Lap-s/2-delta-s} in Section 5 and the Appendix).
It was a surprise to us that its final value is so simple and, at the same time, that the Pohozaev constant $c_1^2(\pi^2+c_2^2)$ also simplifies and becomes $\Gamma(1+s)^2$.

Instead of computing explicitly the constants $c_1$ and $c_2$, an alternative way to obtain the constant in the Pohozaev identity consists of using an
explicit nonlinearity and solution to problem \eqref{eq} in a ball.
The one which is known \cite{G,BGR} is the solution to problem
\[\left\{ \begin{array}{rcll} (-\Delta)^s u &=&1&\textrm{in }B_r(x_0) \\
u&=&0&\textrm{in }\mathbb R^n\backslash B_r(x_0).\end{array}\right.\]
It is given by
\[u(x)=\frac{2^{-2s}\Gamma(n/2)}{\Gamma\left(\frac{n+2s}{2}\right)\Gamma(1+s)}\left(r^2-|x-x_0|^2\right)^s\qquad\textrm{in}\ \ B_r(x_0).\]
From this, it is straightforward to find the constant $\Gamma(1+s)^2$ in the Pohozaev identity; see Remark
\ref{A4} in the Appendix.

Using Theorem \ref{krylov} and Proposition \ref{intparts}, we can also deduce a Pohozaev identity for problem \eqref{eqlin}, that is, allowing the nonlinearity $f$ to depend also on $x$. In this case, the Pohozaev identity reads as follows.

\begin{prop}\label{proppoh} Let $\Omega$ be a bounded and $C^{1,1}$ domain, $f\in C^{0,1}_{\rm loc}(\overline \Omega\times \R)$, $u$ be a bounded solution of \eqref{eqlin}, and $\delta(x)={\rm dist}(x,\partial\Omega)$. Then
\[u/\delta^s|_{\Omega}\in C^{\alpha}(\overline\Omega)\qquad \textrm{for some }\ \alpha\in(0,1),\]
and the following identity holds
\[(2s-n)\int_{\Omega}uf(x,u)dx+2n\int_\Omega F(x,u)dx=\hspace{50mm}\]
\[\hspace{30mm}=\Gamma(1+s)^2\int_{\partial\Omega}\left(\frac{u}{\delta^s}\right)^2(x\cdot\nu)d\sigma-2\int_\Omega x\cdot F_x(x,u)dx,\]
where $F(x,t)=\int_0^tf(x,\tau)d\tau$, $\nu$ is the unit outward normal to $\partial\Omega$ at $x$, and $\Gamma$ is the Gamma function.
\end{prop}

From this, we deduce nonexistence results for problem \eqref{eqlin} with supercritical nonlinearities $f$ depending also on $x$.
This has been done also in \cite{FW} for positive solutions. Our result allows changing sign solutions as well as a slightly larger class of nonlinearities (see Remark \ref{remFW}).

\begin{cor}\label{cornonexistence3} Let $\Omega$ be a bounded, $C^{1,1}$, and star-shaped domain, $f\in C^{0,1}_{\rm loc}(\overline \Omega\times \R)$, and $F(x,t)=\int_0^tf(x,\tau)d\tau$. If
\begin{equation}\label{supercritic2} \frac{n-2s}{2}\,uf(x,t)\geq nF(x,t)+x\cdot F_x(x,t)\qquad\textrm{for all}\ \ x\in\Omega\ \ \textrm{and}\ \ t\in\mathbb R, \end{equation}
then problem \eqref{eqlin} admits no positive bounded solution. Moreover, if the inequality in \eqref{supercritic2} is strict, then \eqref{eqlin} admits no nontrivial bounded solution.
\end{cor}

\begin{rem}\label{remFW} For locally Lipschitz nonlinearities $f$, condition \eqref{supercritic2} is more general than the one required in \cite{FW} for their nonexistence result.
Namely, \cite{FW} assumes that for each $x\in\Omega$ and $t\in\mathbb R$, the map
\begin{equation}\label{condFW}\lambda\mapsto \lambda^{-\frac{n+2s}{n-2s}}f(\lambda^{-\frac{2}{n-2s}}x,\lambda t)\qquad  \textrm{is nondecreasing for }\lambda\in (0,1].\end{equation}
Such nonlinearities automatically satisfy \eqref{supercritic2}.

However, in \cite{FW} they do not need to assume any regularity on $f$ with respect to~$x$.
\end{rem}

The paper is organized as follows. In Section \ref{sec2}, using Propositions \ref{thlaps/2} and \ref{propoperador} (to be established later), we prove Proposition \ref{intparts} (the Pohozaev identity) for strictly star-shaped domains with respect to the origin.
We also establish the nonexistence results for supercritical nonlinearities, and this does not require any result from the rest of the paper.
In Section \ref{sec6} we establish Proposition \ref{thlaps/2}, while in Section \ref{sec7} we prove Proposition \ref{propoperador}.
Section \ref{sec8} establishes Proposition \ref{intparts} for non-star-shaped domains and all its consequences, which include Theorems \ref{thpoh} and \ref{corintparts} and the nonexistence results.
Finally, in the Appendix we compute the constants $c_1$ and $c_2$ appearing in Proposition \ref{thlaps/2}.

\section{Star-shaped domains: Pohozaev identity and nonexistence}
\label{sec2}

In this section we prove Proposition \ref{intparts}
for strictly star-shaped domains.
We say that $\Omega$ is strictly star-shaped if, for some $z_0\in\R^n$,
\begin{equation}\label{starshaped} (x-z_0)\cdot\nu>0\qquad \textrm{for all}\ \ x\in\partial\Omega.\end{equation}
The result for general $C^{1,1}$ domains will be a consequence of this strictly star-shaped case and will be proved in Section \ref{sec8}.

The proof in this section uses two of our results: Proposition \ref{thlaps/2} on the behavior of $(-\Delta)^{s/2}u$ near $\partial\Omega$ and the one dimensional computation of Proposition \ref{propoperador}.

The idea of the proof for the fractional Pohozaev identity is to use the integration by parts formula \eqref{fip} with $v=u_\lambda$, where
\[u_\lambda(x)=u(\lambda x), \quad \lambda>1,\]
and then differentiate the obtained identity (which depends on $\lambda$) with respect to $\lambda$ and evaluate at $\lambda=1$. However, this apparently simple formal procedure requires a quite involved analysis when it is put into practice. The hypothesis that $\Omega$ is star-shaped is crucially used in order that $u_\lambda$, $\lambda>1$, vanishes outside $\Omega$ so that \eqref{fip} holds.

\begin{proof}[Proof of Proposition \ref{intparts} for strictly star-shaped domains]
Let us assume first that $\Omega$ is strictly star-shaped with respect to the origin, that is, $z_0=0$.

Let us prove that
\begin{equation}\label{first}\int_\Omega(x\cdot \nabla u) (-\Delta)^su\, dx=\left.\frac{d}{d\lambda}\right|_{\lambda=1^+}\int_\Omega u_\lambda (-\Delta)^su\, dx,\end{equation}
where $\left.\frac{d}{d\lambda}\right|_{\lambda=1^+}$ is the derivative from the right side at $\lambda=1$.
Indeed, let $g=(-\Delta)^s u$.
By assumption (a) $g$ is defined pointwise in $\Omega$, and by assumption (c) $g\in L^\infty(\Omega)$.
Then, making the change of variables $y=\lambda x$ and using that $\textrm{supp}\,u_\lambda=\frac{1}{\lambda}\Omega\subset \Omega$ since $\lambda>1$, we obtain
\[\begin{split}\left.\frac{d}{d\lambda}\right|_{\lambda =1^+}&\int_\Omega u_\lambda(x)g(x) dx = \lim_{\lambda\downarrow 1}\int_\Omega \frac{u(\lambda x)-u(x)}{\lambda-1}g(x) dx\\
&\qquad=\lim_{\lambda\downarrow 1} \lambda^{-n}\int_{\lambda\Omega} \frac{u(y)-u(y/\lambda)}{\lambda-1} g(y/\lambda)dy\\
&\qquad=\lim_{\lambda\downarrow 1}\int_\Omega \frac{u(y)-u(y/\lambda)}{\lambda-1} g(y/\lambda)dy+\lim_{\lambda\downarrow1}\int_{(\lambda\Omega)\backslash\Omega}\frac{-u(y/\lambda)}{\lambda-1}
g(y/\lambda)dy.
\end{split}\]
By the dominated convergence theorem,
\[\lim_{\lambda\downarrow 1}\int_\Omega \frac{u(y)-u(y/\lambda)}{\lambda-1}g(y/\lambda)\,dy=\int_\Omega (y\cdot\nabla u)g(y)\,dy,\]
since $g\in L^\infty(\Omega)$, $|\nabla u(\xi)|\leq C\delta(\xi)^{s-1}\leq C\lambda^{1-s}\delta(y)^{s-1}$ for all $\xi$ in the segment joining $y$ and $y/\lambda$, and $\delta^{s-1}$ is integrable.
The gradient bound $|\nabla u(\xi)|\leq C\delta(\xi)^{s-1}$ follows from assumption (a) used with $\beta=1$.
Hence, to prove \eqref{first} it remains only to show that
\[\lim_{\lambda\downarrow 1}\int_{(\lambda\Omega)\backslash\Omega}\frac{-u(y/\lambda)}{\lambda-1}g(y/\lambda)dy=0.\]
Indeed, $|(\lambda\Omega)\backslash\Omega|\leq C(\lambda-1)$ and ---by (a)--- $u\in C^s(\R^n)$ and $u\equiv 0$ outside $\Omega$. Hence, $\|u\|_{L^\infty((\lambda\Omega)\backslash\Omega)}\rightarrow0$ as $\lambda\downarrow1$ and \eqref{first} follows.

Now, using the integration by parts formula \eqref{fip} with $v=u_\lambda$,
\begin{eqnarray*}\int_\Omega u_\lambda (-\Delta)^su\,dx
&=& \int_{\mathbb R^n}u_\lambda (-\Delta)^su\,dx \\
&=& \int_{\mathbb R^n}(-\Delta)^{s/2}u_{\lambda}(-\Delta)^{s/2}u\,dx \\
&=& \lambda^s\int_{\mathbb R^n}\left((-\Delta)^{s/2}u\right)(\lambda x)(-\Delta)^{s/2}u(x)dx \\
&=& \lambda^s\int_{\mathbb R^n}w_{\lambda}w\ dx,
\end{eqnarray*}
where
\[w(x)=(-\Delta)^{s/2}u(x)\qquad {\rm and}\qquad w_{\lambda}(x)=w(\lambda x).\]
With the change of variables $y=\sqrt{\lambda}x$ this integral becomes
\[\lambda^s\int_{\mathbb R^n}w_{\lambda}w\, dx=\lambda^{\frac{2s-n}{2}}\int_{\mathbb R^n}w_{\sqrt{\lambda}}w_{1/\sqrt{\lambda}}\,dy,\]
and thus
\[\int_\Omega u_\lambda(-\Delta)^su\, dx=\lambda^{\frac{2s-n}{2}}\int_{\mathbb R^n}w_{\sqrt{\lambda}}w_{1/\sqrt{\lambda}}\,dy.\]
Furthermore, this leads to
\begin{eqnarray}
\int_\Omega(\nabla u\cdot x)(-\Delta)^su\, dx&=&\left.\frac{d}{d\lambda}\right|_{\lambda=1^+}\left\{\lambda^{\frac{2s-n}{2}}\int_{\mathbb R^n}w_{\sqrt{\lambda}}w_{1/\sqrt{\lambda}}\,dy\right\}\nonumber \\
&=&\frac{2s-n}{2}\int_{\mathbb R^n}|(-\Delta)^{s/2}u|^2\,dx +\left.\frac{d}{d\lambda}\right|_{\lambda=1^+}\int_{\mathbb R^n}w_{\sqrt{\lambda}}w_{1/\sqrt{\lambda}}\,dy\nonumber \\
&=&\frac{2s-n}{2}\int_{\Omega}u(-\Delta)^s u\, dx+ \frac 12\left.\frac{d}{d\lambda}\right|_{\lambda=1^+}\int_{\mathbb R^n}w_{{\lambda}}w_{1/{\lambda}}\,dy\label{igualtatpaluego}
.\end{eqnarray}
Hence, it remains to prove that
\begin{equation}\label{derivadaIlambda}-\left.\frac{d}{d\lambda}\right|_{\lambda=1^+}I_\lambda
=\Gamma(1+s)^2\int_{\partial\Omega}\left(\frac{u}{\delta^s}\right)^2(x\cdot\nu)\,d\sigma,\end{equation}
where we have denoted
\begin{equation}\label{Ilambda} I_\lambda=\int_{\mathbb R^n}w_{{\lambda}}w_{1/{\lambda}}\,dy.\end{equation}

Now, for each $\theta\in S^{n-1}$ there exists a unique $r_\theta>0$ such that $r_\theta\theta\in\partial\Omega$.
Write the integral \eqref{Ilambda} in spherical coordinates and use the change of variables $t=r/r_\theta$:
\begin{eqnarray*}
\left.\frac{d}{d\lambda}\right|_{\lambda=1^+}I_\lambda&=&
\left.\frac{d}{d\lambda}\right|_{\lambda=1^+}\int_{S^{n-1}}d\theta\int_0^\infty r^{n-1}w(\lambda r\theta)w\left(\frac{r}{\lambda}\theta\right)dr\\
&=&\left.\frac{d}{d\lambda}\right|_{\lambda=1^+}\int_{S^{n-1}}r_\theta d\theta\int_0^\infty (r_\theta t)^{n-1}w(\lambda r_\theta t\theta)w\left(\frac{r_\theta t}{\lambda}\theta\right)dt\\
&=&\left.\frac{d}{d\lambda}\right|_{\lambda=1^+}\int_{\partial\Omega}(x\cdot \nu)d\sigma(x)\int_0^\infty t^{n-1}w(\lambda tx)w\left(\frac{tx}{\lambda}\right)dt,
\end{eqnarray*}
where we have used that
\[r_\theta^{n-1}\,d\theta = \left(\frac{x}{|x|}\cdot\nu\right)\,d\sigma=\frac{1}{r_\theta}(x\cdot\nu)\,d\sigma\]
with the change of variables $S^{n-1}\rightarrow \partial\Omega$ that maps every point in $S^{n-1}$ to its radial projection on $\partial\Omega$, which is
unique because of the strictly star-shapedness of $\Omega$.

Fix $x_0\in \partial\Omega$ and define
\[\varphi(t)=t^{\frac{n-1}{2}}w\left(t x_0\right)=t^{\frac{n-1}{2}}(-\Delta)^{s/2}u(t x_0).\]
By Proposition \ref{thlaps/2},
\[\varphi(t)=c_1\{\log^-\delta(tx_0)+c_2\chi_{[0,1]}\}v(tx_0)+h_0(t)\]
in $[0,\infty)$, where $v$ is a $C^\alpha(\R^n)$ extension of $u/\delta^s|_\Omega$ and $h_0$ is a $C^{\alpha}([0,\infty))$ function. Next we will modify this expression in order to apply Proposition \ref{propoperador}.

Using that $\Omega$ is $C^{1,1}$ and strictly star-shaped, it is not difficult to see that $\frac{|r-r_\theta|}{\delta(r\theta)}$ is a Lipschitz function of $r$ in $[0,\infty)$ and bounded below by a positive constant (independently of $x_0$). Similarly, $\frac{|t-1|}{\delta(t x_0)}$ and $\frac{\min\{|t-1|,1\}}{\min\{\delta(tx_0),1\}}$ are positive and Lipschitz functions of $t$ in $[0,\infty)$. Therefore,
\[\log^-|t-1|- \log^- \delta(tx_0)\]
is Lipschitz in $[0,\infty)$ as a function of $t$.

Hence, for $t\in[0,\infty)$,
\[\varphi(t)=c_1\{\log^-|t-1|+c_2\chi_{[0,1]}\}v(tx_0)+h_1(t),\]
where $h_1$ is a $C^{\alpha}$ function in the same interval.

Moreover, note that the difference
\[v(tx_0)-v(x_0)\]
is $C^{\alpha}$ and vanishes at $t=1$. Thus,
\[\varphi(t)=c_1\{\log^-|t-1|+c_2\chi_{[0,1]}(t)\}v(x_0)+h(t)\]
holds in all $[0,\infty)$, where $h$ is  $C^{\alpha}$ in $[0,\infty)$ if we slightly decrease $\alpha$ in order to kill the logarithmic singularity.
This is condition (i) of Proposition \ref{propoperador}.

From the expression
\[h(t)=t^{\frac{n-1}{2}}(-\Delta)^{s/2} u \left(t x_0\right)-c_1\{\log^-|t-1|+c_2\chi_{[0,1]}(t)\}v(x_0)\]
and from \eqref{numeret} in Proposition \ref{thlaps/2}, we obtain that $h$ satisfies condition (ii) of Proposition \ref{propoperador} with $\gamma=s/2$.

Moreover, condition (iii) of Proposition \ref{propoperador} is also satisfied.
Indeed, for $x\in \mathbb R^n\backslash(2\Omega)$ we have
\[(-\Delta)^{s/2}u(x)=c_{n,\frac s2}\int_{\Omega}\frac{-u(y)}{|x-y|^{n+s}}dy\]
and hence
\[|\partial_{i}(-\Delta)^{s/2}u(x)|\leq C|x|^{-n-s-1}\ \ \textrm{ and }\ \ |\partial_{ij}(-\Delta)^{s/2}u(x)|\leq C|x|^{-n-s-2}.\]
This yields $|\varphi'(t)|\leq Ct^{\frac{n-1}{2}-n-s-1}\leq Ct^{-2-\gamma}$ and $|\varphi''(t)|\leq Ct^{\frac{n-1}{2}-n-s-2}\leq Ct^{-3-\gamma}$ for $t>2$.

Therefore we can apply Proposition \ref{propoperador} to obtain
\[\left.\frac{d}{d\lambda}\right|_{\lambda=1^+}\int_0^\infty \varphi(\lambda t)\varphi\left(\frac{t}{\lambda}\right)dt=\left(v(x_0)\right)^2c_1^2(\pi^2+c_2^2),\]
and thus
\[\left.\frac{d}{d\lambda}\right|_{\lambda=1^+}\int_0^\infty t^{n-1}w(\lambda tx_0)w\left(\frac{tx_0}{\lambda}\right)dt=\left(v(x_0)\right)^2c_1^2(\pi^2+c_2^2)\]
for each $x_0\in\partial\Omega$.

Furthermore, by uniform convergence on $x_0$ of the limit defining this derivative (see Proposition \ref{lema1} in Section \ref{sec7}), this leads to
\[\left.\frac{d}{d\lambda}\right|_{\lambda=1^+}I_\lambda=c_1^2(\pi^2+c_2^2)\int_{\partial\Omega}(x_0\cdot\nu)
\left(\frac{u}{\delta^s}(x_0)\right)^2dx_0.\]
Here we have used that, for $x_0\in\partial\Omega$, $v(x_0)$ is uniquely defined by continuity as
\[\left(\frac{u}{\delta^s}\right)(x_0)=\lim_{x\rightarrow x_0,\ x\in\Omega} \frac{u(x)}{\delta^s(x)}.\]

Hence, it only remains to prove that
\[c_1^2(\pi^2+c_2^2)=\Gamma(1+s)^2.\]
But
\[c_1=\frac{\Gamma(1+s)\sin\left(\frac{\pi s}{2}\right)}{\pi}\qquad{\rm and}\qquad c_2=\frac{\pi}{\tan\left(\frac{\pi s}{2}\right)},\]
and therefore
\begin{eqnarray*} c_1^2(\pi^2+c_2^2)&=&\frac{\Gamma(1+s)^2\sin^2\left(\frac{\pi s}{2}\right)}{\pi^2}\left(\pi^2+\frac{\pi^2}{\tan^2\left(\frac{\pi s}{2}\right)}\right)\\
&=&\Gamma(1+s)^2\sin^2\left(\frac{\pi s}{2}\right)\left(1+\frac{\cos^2\left(\frac{\pi s}{2}\right)}{\sin^2\left(\frac{\pi s}{2}\right)}\right)\\
&=&\Gamma(1+s)^2.
\end{eqnarray*}

Assume now that $\Omega$ is strictly star-shaped with respect to a point $z_0\neq0$.
Then, $\Omega$ is strictly star-shaped with respect to all points $z$ in a neighborhood of $z_0$.
Then, making a translation and using the formula for strictly star-shaped domains with respect to the origin, we deduce
\begin{equation}\label{nom}\begin{split}\int_\Omega\left\{(x-z)\cdot\nabla u\right\}(-\Delta)^su\, dx=\frac{2s-n}{2}&\int_{\Omega}u(-\Delta)^su\, dx+\\
&-\frac{\Gamma(1+s)^2}{2}\int_{\partial\Omega}\left(\frac{u}{\delta^s}\right)^2(x-z)\cdot\nu\, d\sigma\end{split}\end{equation}
for each $z$ in a neighborhood of $z_0$.
This yields
\begin{equation}\label{eqintparts}\int_\Omega u_{x_i}(-\Delta)^su\, dx=-\frac{\Gamma(1+s)^2}{2}\int_{\partial\Omega}\left(\frac{u}{\delta^s}\right)^2\nu_i\, d\sigma\end{equation}
for $i=1,...,n$. Thus, by adding to \eqref{nom} a linear combination of \eqref{eqintparts}, we obtain
\[\int_\Omega(x\cdot\nabla u)(-\Delta)^su\, dx=\frac{2s-n}{2}\int_{\Omega}u(-\Delta)^su\, dx-\frac{\Gamma(1+s)^2}{2}\int_{\partial\Omega}\left(\frac{u}{\delta^s}\right)^2x\cdot\nu\, d\sigma.\]
\end{proof}

Next we prove the nonexistence results of Corollaries \ref{cornonexistence}, \ref{cornonexistence2}, and \ref{cornonexistence3} for supercritical nonlinearities in star-shaped domains.
Recall that star-shaped means
$x\cdot\nu\ge 0$ for all $x\in\partial\Omega$.
Although these corollaries follow immediately from Proposition \ref{proppoh} ---as we will see in Section \ref{sec8}---, we give here a short proof of their second part, i.e., nonexistence when the inequality \eqref{supercritic} or \eqref{supercritic2} is strict.
That is, we establish the nonexistence of nontrivial solutions for supercritical nonlinearities (not including the critical case).

Our proof follows the method above towards the Pohozaev identity but does not require the full strength of the identity.
In addition, in terms of regularity results for the equation, the proof only needs an easy gradient estimate for solutions $u$.
Namely,
\[|\nabla u|\leq C\delta^{s-1}\ \mbox{ in }\ \Omega,\]
which follows from part (a) of Theorem \ref{krylov}, proved in \cite{RS}.

\begin{proof}[Proof of Corollaries \ref{cornonexistence}, \ref{cornonexistence2}, and \ref{cornonexistence3} for supercritical nonlinearities] We only have to prove Corollary \ref{cornonexistence3}, since Corollaries \ref{cornonexistence} and \ref{cornonexistence2} follow immediately from it by setting $f(x,u)=f(u)$ and $f(x,u)=|u|^{p-1}u$ respectively.

Let us prove that if $\Omega$ is star-shaped and $u$ is a bounded solution of \eqref{eqlin}, then
\begin{equation}
\frac{2s-n}{2}\int_{\Omega}uf(x,u)dx+n\int_\Omega F(x,u)dx-\int_\Omega x\cdot F_x(x,u)dx\geq 0.
\end{equation}
For this, we follow the beginning of the proof of Proposition \ref{intparts} (given above) to obtain \eqref{igualtatpaluego}, i.e., until the identity
\[\int_\Omega(\nabla u\cdot x)(-\Delta)^su\, dx=\frac{2s-n}{2}\int_{\Omega}u(-\Delta)^s u\, dx+\frac 12\left.\frac{d}{d\lambda}\right|_{\lambda=1^+}I_\lambda,\]
where
\[ I_\lambda=\int_{\mathbb R^n}w_{{\lambda}}w_{1/{\lambda}}\,dx,\qquad w(x)=(-\Delta)^{s/2}u(x), \qquad\mbox{and}\qquad w_{\lambda}(x)=w(\lambda x).\]

This step of the proof only need the star-shapedness of $\Omega$ (and not the strictly star-shapedness) and the regularity result $|\nabla u|\leq C\delta^{s-1}$ in $\Omega$, which follows from Theorem \ref{krylov}, proved in \cite{RS}.

Now, since $(-\Delta)^s u=f(x,u)$ in $\Omega$ and
\[(\nabla u\cdot x)(-\Delta)^su=x\cdot \nabla F(x,u)-x\cdot F_x(x,u),\]
by integrating by parts we deduce
\[-n\int_\Omega F(x,u)dx-\int_\Omega x\cdot F_x(x,u)dx=\frac{2s-n}{2}\int_{\Omega}uf(x,u)dx+\frac 12\left.\frac{d}{d\lambda}\right|_{\lambda=1^+}I_\lambda.\]
Therefore, we only need to show that
\begin{equation}\label{negatiu}\left.\frac{d}{d\lambda}\right|_{\lambda=1^+}I_\lambda\leq0.\end{equation}
But applying H\"older's inequality, for each $\lambda>1$ we have
\[I_\lambda\leq\|w_\lambda\|_{L^2(\mathbb R^n)}\|w_{1/{\lambda}}\|_{L^2(\R^n)}=\|w\|_{L^2(\R^n)}^2=I_1,\]
and \eqref{negatiu} follows.
\end{proof}

\begin{rem} For this nonexistence result the regularity of the domain $\Omega$ is only used for the estimate $|\nabla u|\leq C\delta^{s-1}$.
This estimate only requires  $\Omega$ to be Lipschitz and satisfy an exterior ball condition; see \cite{RS}. In particular, our nonexistence result for supercritical nonlinearities applies to any convex domain, such as a square for instance.
\end{rem}

\begin{rem}\label{remR^n}
When $\Omega=\R^n$ or when $\Omega$ is a star-shaped domain with respect to infinity, there are two recent nonexistence results for subcritical nonlinearities. They use the method of moving spheres to prove nonexistence of bounded positive solutions in these domains.
The first result is due to A. de Pablo and U. S\'anchez \cite{PS}, and they obtain nonexistence of bounded positive solutions to $(-\Delta)^s u=u^p$ in all of $\R^n$, whenever $s>1/2$ and $1<p<\frac{n+2s}{n-2s}$.
The second result, by M. Fall and T. Weth \cite{FW}, gives nonexistence of bounded positive solutions of \eqref{eqlin} in star-shaped domains with respect to infinity for subcritical nonlinearities.

Our method in the previous proof can also be used to prove nonexistence results for problem \eqref{eqp} in star-shaped domains with respect to infinity or in the whole $\R^n$.
However, to ensure that the integrals appearing in the proof are well defined, one must assume some decay on $u$ and $\nabla u$. For instance, in the supercritical case $p> \frac{n+2s}{n-2s}$ we obtain that the only solution to  $(-\Delta)^s u=u^p$ in all of $\R^n$ decaying as
\[|u|+|x\cdot\nabla u|\leq \frac{C}{1+|x|^{\beta}},\]
with $\beta>\frac{n}{p+1}$, is $u\equiv 0$.

In the case of the whole $\R^n$, there is an alternative proof of the nonexistence of solutions which decay fast enough at infinity. It consists of using a Pohozaev identity in all of $\R^n$, that is easily deduced from the pointwise equality
\[(-\Delta)^s(x\cdot \nabla u)=2s(-\Delta)^su+x\cdot \nabla (-\Delta)^su.\]

The classification of solutions in the whole $\R^n$ for the critical exponent $p=\frac{n+2s}{n-2s}$ was obtained by W. Chen, C. Li, and B. Ou in \cite{CLO}. They are of the form
\[u(x)=c\left(\frac{\mu}{\mu^2+|x-x_0|^2}\right)^{\frac{n-2s}{2}},\]
where $\mu$ is any positive parameter and $c$ is a constant depending on $n$ and $s$.
\end{rem}

\section{Behavior of $(-\Delta)^{s/2}u$ near $\partial\Omega$}\label{sec6}

The aim of this section is to prove Proposition \ref{thlaps/2}. We will split this proof into two propositions. The first one is the following, and compares the behavior of $(-\Delta)^{s/2}u$ near $\partial\Omega$ with the one of $(-\Delta)^{s/2}\delta_0^s$, where $\delta_0(x)=\textrm{dist}(x,\partial\Omega)\chi_\Omega(x)$.

\begin{prop}\label{proplaps2} Let $\Omega$ be a bounded and $C^{1,1}$ domain, $u$ be a function satisfying {\rm (b)} in Proposition \ref{intparts}.
Then, there exists a $C^{\alpha}(\R^n)$ extension $v$ of $u/\delta^s|_\Omega$ such that
\[(-\Delta)^{s/2}u(x)=(-\Delta)^{s/2}\delta_0^s(x)v(x)+h(x)\ \mbox{ in }\ \R^n,\]
where $h\in C^{\alpha}(\R^n)$.
\end{prop}

Once we know that the behavior of $(-\Delta)^{s/2}u$ is comparable to the one of $(-\Delta)^{s/2}\delta_0^s$, Proposition \ref{thlaps/2} reduces to the following result, which gives the behavior of $(-\Delta)^s\delta_0^s$ near $\partial\Omega$.

\begin{prop}\label{prop:Lap-s/2-delta-s} Let $\Omega$ be a bounded and $C^{1,1}$ domain, $\delta(x)={\rm dist}(x,\partial\Omega)$, and $\delta_0=\delta\chi_\Omega$.
Then,
\[(-\Delta)^{s/2}\delta_0^s(x)=c_1\left\{\log^- \delta(x)+c_2\chi_\Omega(x)\right\}+h(x)\ \mbox{ in }\ \R^n,\]
where $c_1$ and $c_2$ are constants, $h$ is a $C^{\alpha}(\R^n)$ function, and $\log^-t=\min\{\log t,0\}$.
The constants $c_1$ and $c_2$ are given by
\[c_1=c_{1,\frac s2}\qquad\mbox{and}\qquad c_2=\int_0^{\infty}\left\{\frac{1-z^s}{|1-z|^{1+s}}+\frac{1+z^s}{|1+z|^{1+s}}\right\}dz,\]
where $c_{n,s}$ is the constant appearing in the singular integral expression \eqref{laps} for $(-\Delta)^{s}$ in dimension $n$.
\end{prop}

The fact that the constants $c_1$ and $c_2$ given by Proposition \ref{prop:Lap-s/2-delta-s} coincide with the ones from Proposition \ref{thlaps/2} is proved in the Appendix.

In the proof of Proposition \ref{proplaps2} we need to compute $(-\Delta)^{s/2}$ of the product $u=\delta_0^s v$. For it, we will use the following elementary identity, which can be derived from \eqref{laps}:
\[(-\Delta)^s(w_1w_2)=w_1(-\Delta)^sw_2+w_2(-\Delta)^sw_1-I_s(w_1,w_2),\]
where
\begin{equation}\label{Is}I_s(w_1,w_2)(x)= c_{n,s}\textrm{PV}\int_{\R^n} \frac{\bigl(w_1(x)-w_1(y)\bigr)\bigl(w_2(x)-w_2(y)\bigr)}{|x-y|^{n+2s}}\,dy.
\end{equation}
Next lemma will lead to a H\"older bound for $I_s(\delta_0^s,v)$.

\begin{lem}\label{lem-boundI2}
Let $\Omega$ be a bounded domain and $\delta_0={\rm dist}(x,\R^n\setminus\Omega)$. Then, for each $\alpha\in(0,1)$ the following a priori bound holds
\begin{equation}\label{eq:bound-I2}
\|I_{s/2}(\delta_0^s,w)\|_{C^{\alpha/2}(\R^n)}\le C[w]_{C^\alpha(\R^n)},
\end{equation}
where the constant $C$ depends only on $n$, $s$, and $\alpha$.
\end{lem}

\begin{proof} Let $x_1, x_2\in \R^n$. Then,
\[|I_{s/2} (\delta_0^s,w)(x_1)-I_{s/2}(\delta_0^s,w)(x_2)| \le c_{n,\frac s2}(J_1 + J_2),\]
where
\[J_1= \int_{\R^n}\frac{\bigl|w(x_1)-w(x_1+z)-w(x_2)+w(x_2+z)\bigr|
\bigl|\delta_0^s(x_1)-\delta_0^s(x_1+z)\bigr|}{|z|^{n+s}}\,dz  \]
and
\[J_2= \int_{\R^n}\frac{\bigl|w(x_2)-w(x_2+z)\bigr|
\bigl|\delta_0^s(x_1)-\delta_0^s(x_1+z)-\delta_0^s(x_2)+\delta_0^s(x_2+z)\bigr|}{|z|^{n+s}}\,dz \,.\]
Let $r=|x_1-x_2|$. Using that $\|\delta_0^s\|_{C^{s}(\R^n)}\le 1$ and ${\rm supp}\,\delta_0^s = \overline\Omega$,
\[
\begin{split}
J_1&\le \int_{\R^n}\frac{\bigl|w(x_1)-w(x_1+z)-w(x_2)+w(x_2+z)\bigr|\min\{|z|^s,({\rm diam}\, \Omega)^s\} }{|z|^{n+s}}\,dz \\
&\le  C\int_{\R^n} \frac{[w]_{C^\alpha(\R^n)}r^{\alpha/2}|z|^{\alpha/2}\min\{|z|^s,1\} }{|z|^{n+s}}\,dz \\
&\le  C r^{\alpha/2}[w]_{C^\alpha(\R^n)}\,.
\end{split}
\]
Analogously,
\[
J_2\le  C r^{\alpha/2}[w]_{C^\alpha(\R^n)}\,.
\]

The bound for $\|I_{s/2}(\delta_0^s,w)\|_{L^\infty(\R^n)}$ is obtained with a similar argument, and hence \eqref{eq:bound-I2} follows.
\end{proof}

Before stating the next result, we need to introduce the following weighted H\"older norms; see Definition 1.3 in \cite{RS}.

\begin{defi}\label{definorm} Let $\beta>0$ and $\sigma\ge -\beta$. Let $\beta=k+\beta'$, with $k$ integer and $\beta'\in (0,1]$.
For $w\in C^{\beta}(\Omega)=C^{k,\beta'}(\Omega)$, define the seminorm
\[ [w]_{\beta;\Omega}^{(\sigma)}= \sup_{x,y\in \Omega} \biggl(\min\{\delta(x),\delta(y)\}^{\beta+\sigma} \frac{|D^{k}w(x)-D^{k}w(y)|}{|x-y|^{\beta'}}\biggr).\]
For $\sigma>-1$, we also define  the norm $\|\cdot\|_{\beta;\Omega}^{(\sigma)}$ as follows: in case that $\sigma\ge0$,
\[ \|w\|_{\beta;\Omega}^{(\sigma)} = \sum_{l=0}^k \sup_{x\in \Omega} \biggl(\delta(x)^{l+\sigma} |D^l w(x)|\biggr) + [w]_{\beta;\Omega}^{(\sigma)}\,,\]
while for $-1<\sigma<0$,
\[\|w\|_{\beta;\Omega}^{(\sigma)} = \|w\|_{C^{-\sigma}(\overline \Omega)}+\sum_{l=1}^k \sup_{x\in \Omega} \biggl(\delta(x)^{l+\sigma} |D^l w(x)|\biggr) + [w]_{\beta;\Omega}^{(-\alpha)}.\]
\end{defi}

The following lemma, proved in \cite{RS}, will be used in the proof of Proposition \ref{proplaps2} below ---with $w$ replaced by $v$--- and also at the end of this section in the proof of Proposition \ref{thlaps/2} ---with $w$ replaced by $u$.

\begin{lem}[{\cite[Lemma 4.3]{RS}}]
\label{cosaiscalpha}
Let $\Omega$ be a bounded domain and $\alpha$ and $\beta$ be such that $\alpha\le s<\beta$ and $\beta-s$ is not an integer. Let $k$ be an integer such that $\beta= k+\beta'$ with $\beta'\in(0,1]$. Then,
\begin{equation}\label{eq:bound-lap-s/2}
[(-\Delta)^{s/2}w]_{\beta-s;\Omega}^{(s-\alpha)}\le C\bigl( \|w\|_{C^\alpha(\R^n)}+ \|w\|_{\beta;\Omega}^{(-\alpha)}\bigl)
\end{equation}
for all $w$ with finite right hand side.
The constant $C$ depends only on $n$, $s$, $\alpha$, and $\beta$.
\end{lem}

Before proving Proposition \ref{proplaps2}, we give an extension lemma ---see \cite[Theorem~1, Section 3.1]{EG} where the case $\alpha=1$ is proven in full detail.

\begin{lem} \label{ext}
Let $\alpha\in(0,1]$ and $V\subset\R^n$ a bounded domain. There exists a (nonlinear) map $E:C^{0,\alpha}(\overline V)\rightarrow C^{0,\alpha}(\R^n)$ satisfying
\[ E(w)\equiv w \quad \mbox{in }\overline V,\ \ \ [E(w)]_{C^{0,\alpha}(\R^n)}\le [w]_{C^{0,\alpha}(\overline V)},\ \ \ \mbox{and}\ \ \ \|E(w)\|_{L^\infty(\R^n)}\le \|w\|_{L^\infty(V)}\]
for all $w\in C^{0,\alpha}(\overline V)$.
\end{lem}

\begin{proof}
It is immediate to check that
\[E(w)(x)=\min\left\{\min_{z\in \overline V}\left\{w(z)+
[w]_{C^{\alpha}(\overline V)}|z-x|^\alpha\right\},\|w\|_{L^\infty(V)}\right\}\]
satisfies the conditions since, for all $x,y,z$ in $\R^n$,
\[|z-x|^\alpha \le |z-y|^\alpha+|y-x|^\alpha\,.\]
\end{proof}

Now we can give the

\begin{proof}[Proof of Proposition \ref{proplaps2}]
Since $u/\delta^s|_\Omega$ is $C^{\alpha}(\overline\Omega)$ ---by hypothesis (b)--- then by Lemma \ref{ext}
there exists a $C^\alpha(\R^n)$ extension $v$ of $u/\delta^s|_\Omega$.

Then, we have that
\[(-\Delta)^{s/2}u(x)=v(x)(-\Delta)^{s/2}\delta_0^s(x)+\delta_0(x)^s(-\Delta)^{s/2}v(x)-I_{s/2}(v,\delta_0^s),\]
where
\[I_{s/2}(v,\delta_0^s)= c_{n,\frac s2}\int_{\R^n} \frac{\bigl(v(x)-v(y)\bigr)\bigl(\delta_0^s(x)-\delta_0^s(y)\bigr)}{|x-y|^{n+s}}\,dy \,,\]
as defined in \eqref{Is}.
This equality is valid in all of $\R^n$ because $\delta_0^s\equiv0$ in $\R^n\backslash \Omega$ and $v\in C^{\alpha+s}$ in $\Omega$ ---by hypothesis (b).
Thus, we only have to see that $\delta_0^s(-\Delta)^{s/2}v$ and $I_{s/2}(v,\delta_0^s)$ are $C^{\alpha}(\R^n)$ functions.

For the first one we combine assumption (b) with $\beta = s+\alpha<1$ and Lemma \ref{cosaiscalpha}. We obtain
\begin{equation}\label{paloma}
\|(-\Delta)^{s/2}v\|_{\alpha;\Omega}^{(s-\alpha)} \le C,
\end{equation}
and this yields $\delta_0^s(-\Delta)^{s/2}v \in C^{\alpha}(\R^n)$.
Indeed, let $w=(-\Delta)^{s/2}v$. Then, for all $x,y\in \Omega$ such that $y\in B_R(x)$, with $R=\delta(x)/2$, we have
\[\frac{|\delta^s(x)w(x)-\delta^s(y)w(y)|}{|x-y|^\alpha}\leq \delta(x)^s\frac{|w(x)-w(y)|}{|x-y|^\alpha}+|w(x)|\frac{|\delta^s(x)-\delta^s(y)|}{|x-y|^\alpha}.\]
Now, since \[|\delta^s(x)-\delta^s(y)|\leq C R^{s-\alpha} |x-y|^\alpha \le C \min\{\delta(x),\delta(y)\}^{s-\alpha} |x-y|^\alpha,\]
using  \eqref{paloma} and recalling Definition \ref{definorm} we obtain
\[\frac{|\delta^s(x)w(x)-\delta^s(y)w(y)|}{|x-y|^\alpha} \le C \quad \mbox{whenever }y\in B_R(x)\,,\ R=\delta(x)/2.\]
This bound can be extended to all $x,y\in \Omega$, since the domain is regular, by using a dyadic chain of balls; see for instance the proof of Proposition 1.1 in \cite{RS}.

The second bound, that is,
\[\|I_{s/2}(v,\delta_0^s)\|_{C^{\alpha}(\R^n)}\leq C,\]
follows from assumption (b) and Lemma \ref{lem-boundI2} (taking a smaller $\alpha$ if necessary).
\end{proof}

To prove Proposition \ref{prop:Lap-s/2-delta-s} we need some preliminaries.

Fixed $\rho_0>0$, define $\phi\in C^{s}(\R)$ by
\begin{equation}\label{phi}
\phi(x)= x^s  \chi_{(0,\rho_0)}(x)+\rho_0^s \chi_{(\rho_0,+\infty)}(x).
\end{equation}
This function $\phi$ is a truncation of the $s$-harmonic function $x_+^s$.
We need to introduce $\phi$ because the growth at infinity of $x_+^s$ prevents us
from computing its $(-\Delta)^{s/2}$.

\begin{lem}\label{claim1} Let $\rho_0>0$, and let $\phi:\R\rightarrow\R$ be given by \eqref{phi}. Then, we have
\[(-\Delta)^{s/2} \phi (x) = c_1\{\log |x| + c_2\chi_{(0,\infty)}(x)\} + h(x)\]
for $x\in(-\rho_0/2,\rho_0/2)$, where $h\in C^{s}([-\rho_0/2,\rho_0/2])$. The constants $c_1$ and $c_2$ are given by
\[c_1=c_{1,\frac s2}\qquad{\rm and}\qquad c_2=\int_0^{\infty}\left\{\frac{1-z^s}{|1-z|^{1+s}}+\frac{1+z^s}{|1+z|^{1+s}}\right\}dz,\]
where $c_{n,s}$ is the constant appearing in the singular integral expression \eqref{laps} for $(-\Delta)^{s}$ in dimension $n$.
\end{lem}

\begin{proof}
If $x<\rho_0$,
\[(-\Delta)^{s/2}\phi(x) =
 c_{1,\frac s2}\biggl(\int_{-\infty}^{\rho_0} \frac{x_+^s-y_+^s}{|x-y|^{1+s}}\,dy
 +  \int_{\rho_0}^{\infty}\frac{x_+^s-\rho_0^s}{|x-y|^{1+s}}\,dy\biggr).\]
We need to study the first integral:
\begin{equation}\label{integrals}
J(x)=\int_{-\infty}^{\rho_0}\frac{x_+^s-y_+^s}{|x-y|^{1+s}}\,dy =
\begin{cases}\displaystyle
\ J_1(x)= \int_{-\infty}^{\rho_0/x}\frac{1-z_+^s}{|1-z|^{1+s}}\,dz &\mbox{if }x>0\\
&\\
\displaystyle
\ J_2(x)=\int_{-\infty}^{\rho_0/|x|}\frac{-z_+^s}{|1+z|^{1+s}} \,dz&     \mbox{if }x<0\,,
\end{cases}
\end{equation}
since
\begin{equation}\label{dosestrelles}
(-\Delta)^{s/2}\phi(x)-c_1J(x)=
c_1\int_{\rho_0}^{\infty}\frac{x_+^s-\rho_0^s}{|x-y|^{1+s}}\,dy\end{equation}
belongs to $C^{s}([-\rho_0/2,\rho_0/2])$ as a function of $x$.

Using L'H\^opital's rule we find that
\[\lim_{x\downarrow0} \frac{J_1(x)} {\log |x|}=\lim_{x\uparrow0}
\frac{J_2(x)}{\log |x|}=1.\]
Moreover,
\[\begin{split}
\lim_{x\downarrow0} x^{1-s}\biggl(J_1'(x)-\frac{1}{x}\biggr)&=
\lim_{x\downarrow0} x^{1-s}\biggl(-\frac{\rho_0}{x^2}\frac{1-(\rho_0/x)^s}{((\rho_0/x)-1)^{1+s}}-\frac{1}{x}\biggr)\\
&=\rho_0^{-s}\lim_{y\downarrow0} y^{1-s} \biggl(\frac{1-y^s}{y(1-y)^{1+s}}-\frac{(1-y)^{1+s}}{y(1-y)^{1+s}}\biggr)\\
&=\rho_0^{-s}\lim_{y\downarrow0} \frac{1-y^s-(1-y)^{1+s}}{y^s}= -\rho_0^{-s}
\end{split}
\]
and
\[\begin{split}
\lim_{x\uparrow0} (-x)^{1-s}\biggl(J_2'(x)-\frac{1}{x}\biggr)&=
\lim_{x\uparrow0} (-x)^{1-s}\biggl(\frac{\rho_0}{x^2}\frac{-(-\rho_0/x)^s}
{(1+(-\rho_0/x))^{1+s}}-\frac{1}{-x}\biggr)\\
&=\rho_0^{-s}\lim_{y\downarrow0} y^{1-s} \biggl(\frac{-1}{y(1+y)^{1+s}}+\frac{(1+y)^{1+s}}{y(1+y)^{1+s}}\biggr)\\
&=\rho_0^{-s}\lim_{y\downarrow0} \frac{(1+y)^{1+s}-1}{y^s}= 0 \,.
\end{split}
\]

Therefore,
\[(J_1(x)-\log|x|)' \le C |x|^{s-1}\quad \mbox{in }(0,\rho_0/2]\]
and
\[(J_2(x)-\log|x|)' \le C |x|^{s-1}\quad \mbox{in }[-\rho_0/2,0),\]
and these gradient bounds yield
\[(J_1-\log|\cdot|)\in C^{s}([0,\rho_0/2])\quad \mbox{and}\quad
(J_2-\log|\cdot|)\in C^{s}([-\rho_0/2,0]).\]
However, these two H\"older functions do not have the same value at $0$.
Indeed,
\[
\begin{split}
\lim_{x\downarrow 0} \bigl\{(J_1(x)-\log|x|)-(&J_2(-x)-\log|-x|)\bigr\}
=\lim_{x\downarrow 0}\left\{J_1(x)-J_2(-x)\right\}\\
&=\int_{-\infty}^{\infty}\left\{\frac{1-z_+^s}{|1-z|^{1+s}}+\frac{z_+^s}{|1+z|^{1+s}}\right\}dz\\
&=\int_{0}^{\infty}\left\{\frac{1-z^s}{|1-z|^{1+s}}+\frac{1+z^s}{|1+z|^{1+s}}\right\}dz=c_2.
\end{split}
\]
Hence, the function $J(x)-\log|x|-c_2\chi_{(0,\infty)}(x)$, where $J$ is defined by \eqref{integrals}, is $C^s([-\rho_0/2,\rho_0/2])$.
Recalling \eqref{dosestrelles}, we obtain the result.
\end{proof}

Next lemma will be used to prove Proposition \ref{prop:Lap-s/2-delta-s}. Before stating it, we need the following

\begin{rem}\label{remrho0}
From now on in this section, $\rho_0>0$ is a small constant depending only on $\Omega$, which we assume to be a bounded $C^{1,1}$ domain.
Namely, we assume that that every point on $\partial\Omega$ can be touched from both inside and outside $\Omega$ by balls of radius $\rho_0$.
In other words, given $x_0\in \partial\Omega$, there are balls of radius $\rho_0$, $B_{\rho_0}(x_1)\subset \Omega$ and $B_{\rho_0}(x_2)\subset\R^n\setminus \Omega$, such that $\overline{B_{\rho_0}(x_1)}\cap\overline{B_{\rho_0}(x_2)}=\{x_0\}$.
A useful observation is that all points $y$ in the segment that joins $x_1$ and $x_2$ ---through $x_0$--- satisfy
$\delta(y)= |y-x_0|$.
\end{rem}

\begin{lem}\label{claim2} Let $\Omega$ be a bounded $C^{1,1}$ domain, $\delta(x)={\rm dist}(x,\partial\Omega)$, $\delta_0=\delta\chi_\Omega$, and $\rho_0$ be given by Remark \ref{remrho0}. Fix $x_0\in\partial\Omega$, and define
\[\phi_{x_0}(x)=\phi\left(-\nu(x_0)\cdot(x-x_0)\right)\]
and
\begin{equation}\label{Sx0}
S_{x_0}=\{x_0+t\nu(x_0),\ t\in(-\rho_0/2,\rho_0/2)\},
\end{equation}
where $\phi$ is given by \eqref{phi} and $\nu(x_0)$ is the unit outward normal to $\partial\Omega$ at $x_0$.
Define also $w_{x_0}=\delta_0^s-\phi_{x_0}$.

Then, for all $x\in S_{x_0}$,
\[|(-\Delta)^{s/2} w_{x_0}(x)-(-\Delta)^{s/2} w_{x_0}(x_0)| \le C|x-x_0|^{s/2},\]
where $C$ depends only on $\Omega$ and $\rho_0$ (and not on $x_0$).
\end{lem}

\begin{proof} We denote $w=w_{x_0}$.
Note that, along $S_{x_0}$, the distance to $\partial\Omega$ agrees with the distance to the tangent plane to $\partial\Omega$ at $x_0$; see Remark \ref{remrho0}.
That is, denoting $\delta_\pm=(\chi_\Omega-\chi_{\R^n\backslash \Omega})\delta$ and $d(x)=-\nu(x_0)\cdot(x-x_0)$, we have $\delta_\pm(x)=d(x)$ for all $x\in S_{x_0}$.
Moreover, the gradients of these two functions also coincide on $S_{x_0}$, i.e., $\nabla\delta_\pm(x)=-\nu(x_0)=\nabla d(x)$ for all $x\in S_{x_0}$.

Therefore, for all $x\in S_{x_0}$ and $y\in B_{\rho_0/2}(0)$, we have
\[|\delta_{\pm}(x+y)-d(x+y)|\le C|y|^2\]
for some $C$ depending only on $\rho_0$. Thus, for all $x\in S_{x_0}$ and $y\in B_{\rho_0/2}(0)$,
\begin{equation}\label{w(x+y)}|w(x+y)|= |(\delta_{\pm}(x+y))_+^s-(d(x+y))_+^s| \le C|y|^{2s},\end{equation}
where $C$ is a constant depending on $\Omega$ and $s$.

On the other hand, since $w\in C^{s}(\R^n)$, then
\begin{equation}\label{w(x+y)2}|w(x+y)-w(x_0+y)|\leq C|x-x_0|^s.\end{equation}

Finally, let $r<\rho_0/2$ to be chosen later. For each $x\in S_{x_0}$ we have
\[
\begin{split}
|(-\Delta)^{s/2} &w(x)-(-\Delta)^{s/2} w(x_0)|
\le C\int_{\R^n} \frac{|w(x+y)-w(x_0+y)|}{|y|^{n+s}}\,dy\\
&\le C\int_{B_r}\frac{|w(x+y)-w(x_0+y)|}{|y|^{n+s}}\,dy
+C\int_{\R^n\setminus B_r}\frac{|w(x+y)-w(x_0+y)|}{|y|^{n+s}}\,dy\\
&\le C\int_{B_r}\frac{|y|^{2s}}{|y|^{n+s}}\,dy+
C\int_{\R^n\setminus B_r}\frac{|x-x_0|^s}{|y|^{n+s}}\,dy\\
&= C(r^s +|x-x_0|^sr^{-s})\,,
\end{split}
\]
where we have used \eqref{w(x+y)} and \eqref{w(x+y)2}.
Taking $r=|x-x_0|^{1/2}$ the lemma is proved.
\end{proof}

The following is the last ingredient needed to prove Proposition \ref{prop:Lap-s/2-delta-s}.

\begin{claim}\label{claimMMM}
Let $\Omega$ be a bounded $C^{1,1}$ domain, and $\rho_0$ be given by Remark \ref{remrho0}.
Let $w$ be a function satisfying, for some $K>0$,
\begin{itemize}
\item[(i)] $w$ is locally Lipschitz in $\{x\in\R^n \,:\,0<\delta(x)<\rho_0\}$ and
\[|\nabla w(x)| \le K\delta(x)^{-M}\ \mbox{ in }\ \{x\in\R^n \,:\,0<\delta(x)<\rho_0\}\]
for some $M >0$.
\item[(ii)] There exists $\alpha>0$ such that
\[|w(x)-w(x^*)|\le K\delta(x)^\alpha \ \mbox{ in }\ \{x\in\R^n \,:\,0<\delta(x)<\rho_0\},\]
where $x^*$ is the unique point on $\partial \Omega$ satisfying $\delta(x)= |x-x^*|$.
\item[(iii)] For the same $\alpha$, it holds
\[
\|w\|_{C^{\alpha}\left( {\{\delta\geq\rho_0\}}\right)}\le K.
\]
\end{itemize}

Then, there exists $\gamma>0$, depending only on $\alpha$ and $M$, such that
\begin{equation}\label{claimMMMeq}
\|w\|_{C^{\gamma}(\R^n)}\le C K,
\end{equation}
where $C$ depends only on $\Omega$.
\end{claim}

\begin{proof} First note that from (ii) and (iii) we deduce that $\|w\|_{L^\infty(\R^n)}\leq CK$.
Let $\rho_1\leq\rho_0$ be a small positive constant to be chosen later.
Let $x,y \in \{\delta\leq\rho_0\} $, and $r=|x-y|$.

If $r\geq\rho_1$, then
\[\frac{|w(x)-w(y)|}{|x-y|^\gamma}\leq \frac{2\|w\|_{L^\infty(\R^n)}}{\rho_1^\gamma}\leq CK.\]

If $r<\rho_1$, consider
\[x' =  x^* +  \rho_0 r^\beta \nu(x^*)\ \mbox{ and }\ y' =  y^* + \rho_0 r^\beta \nu(y^*),\]
where $\beta\in(0,1)$ is to be determined later. Choose $\rho_1$ small enough so that the segment joining $x'$ and $y'$ contained in the set $\{\delta>\rho_0 r^\beta/2 \}$.
Then, by (i),
\begin{equation}\label{1209}|w(x')-w(y')|\le CK  (\rho_0 r^{\beta}/2) ^{-M} |x'-y'| \le C r^{1-\beta M}.\end{equation}
Thus, using (ii) and \eqref{1209},
\[
\begin{split}
|w(x)-w(y)|&\le |w(x)-w(x^*)|+ |w(x^*)-w(x')|+
\\&\qquad +|w(y)-w(y^*)|+ |w(y^*)-w(y')|+ |w(x')-w(y')|\\
&\le K \delta(x)^\alpha + K\delta(y)^\alpha +2K (\rho_0 r^\beta)^\alpha + CK r^{1-\beta M}.
\end{split}
\]
Taking $\beta<1/M$ and $\gamma= \min\{\alpha\beta, 1-\beta M\}$, we find
\[|w(x)-w(y)|\le CKr^\gamma = CK|x-y|^\gamma.\]
This proves
\[[w]_{C^{\gamma}\left({\{\delta\leq\rho_0\}}\right)}\le CK.\]
To obtain the bound \eqref{claimMMMeq} we combine the previous seminorm estimate with (iii).
\end{proof}

Finally, we give the proof of Proposition \ref{prop:Lap-s/2-delta-s}.

\begin{proof}[Proof of Proposition \ref{prop:Lap-s/2-delta-s}]
Let
\[h(x)= (-\Delta)^{s/2}\delta_0^s(x)-c_1\left\{\log^- \delta(x)+c_2\chi_\Omega(x)\right\}.\]
We want to prove that $h\in C^{\alpha}(\R^n)$ by using Claim \ref{claimMMM}.

On one hand, by Lemma \ref{claim1}, for all $x_0\in\partial\Omega$ and for all $x\in S_{x_0}$, where $S_{x_0}$ is defined by \eqref{Sx0}, we have
\[h (x) = (-\Delta)^{s/2}\delta_0^s(x)- (-\Delta)^{s/2} \phi_{x_0}(x) + \tilde h\bigl(\nu(x_0)\cdot(x-x_0)\bigr),\]
where $\tilde h$ is the $C^{s}([-\rho_0/2,\rho_0/2])$ function from Lemma \ref{claim1}. Hence, using Lemma \ref{claim2}, we find
\[|h(x)-h(x_0)|\le C|x-x_0|^{s/2}\quad \mbox{for all }x\in S_{x_0}\]
for some constant independent of $x_0$.

Recall that for all $x\in S_{x_0}$ we have $x^*=x_0$, where $x^*$ is the unique point on $\partial \Omega$ satisfying $\delta(x)= |x-x^*|$. Hence,
\begin{equation}\label{eq:pf-prop-calpha-bound1}
|h(x)-h(x^*)|\le C|x-x^*|^{s/2}\quad \mbox{ for all }x\in\{\delta<\rho_0/2\}\,.
\end{equation}

Moreover,
\begin{equation}\label{eq:pf-prop-bound-faraway}
\|h\|_{C^{\alpha}(\{\delta\ge\rho_0/2\})} \le C
\end{equation}
for all $\alpha\in (0,1-s)$, where $C$ is a constant depending only on $\alpha$, $\Omega$ and $\rho_0$.
This last bound is found using that $\|\delta_0^s\|_{C^{0,1}(\{\delta\ge\rho_0/2\})} \le C$, which yields
\[\|(-\Delta)^{s/2}\delta_0^s\|_{C^{\alpha}(\{\delta\ge\rho_0\})} \le C\]
for $\alpha<1-s$.

On the other hand, we claim now that if $x\notin \partial\Omega$ and $\delta(x)<\rho_0/2$, then
\begin{equation}\label{eq:pf-prop-grad-bound}
|\nabla h(x)|\le |\nabla(-\Delta)^{s/2}\delta_0^s(x)| + c_1|\delta(x)|^{-1}\le C|\delta(x)|^{-n-s}.
\end{equation}
Indeed, observe that $\delta_0^s\equiv0$ in $\R^n\backslash\Omega$, $|\nabla\delta_0^s|\leq C\delta_0^{s-1}$ in $\Omega$, and $|D^2\delta_0^s|\leq C\delta_0^{s-2}$ in $\Omega_{\rho_0}$.
Then, $r= \delta(x)/2$,
\[\begin{split}
|(-\Delta)^{s/2}\nabla\delta_0^s(x)| &\le
C\int_{\R^n} \frac{|\nabla\delta_0^s(x) - \nabla\delta_0^s(x+y)|}{|y|^{n+s}}\,dy\\
&\le C\int_{B_r} \frac{C r^{s-2}|y|\,dy}{|y|^{n+s}}+ C\int_{\R^n\setminus B_r} \left(\frac{|\nabla\delta_0^s(x)|}{|y|^{n+s}}+ \frac{|\nabla\delta_0^s(x+y)|}{r^{n+s}}\right)dy\\
&\leq \frac Cr+\frac Cr+\frac{C}{r^{n+s}}\int_{\R^n} \delta_0^{s-1}\leq \frac{C}{r^{n+s}},
\end{split}
\]
as claimed.

To conclude the proof, we use bounds \eqref{eq:pf-prop-calpha-bound1}, \eqref{eq:pf-prop-bound-faraway}, and \eqref{eq:pf-prop-grad-bound} and Claim \ref{claimMMM}.
\end{proof}

To end this section, we give the

\begin{proof}[Proof of Proposition \ref{thlaps/2}]
The first part follows from Propositions \ref{proplaps2} and \ref{prop:Lap-s/2-delta-s}.
The second part follows from Lemma \ref{cosaiscalpha} with $\alpha=s$ and $\beta\in(s,1+2s)$.
\end{proof}

\section{The operator $-\left.\frac{d}{d\lambda}\right|_{\lambda=1^+}\int_{\mathbb R}w_\lambda w_{1/\lambda}$}
\label{sec7}

The aim of this section is to prove Proposition \ref{propoperador}.
In other words, we want to evaluate the operator
\begin{equation}\label{mathfrakI}
\mathfrak{I}(w)=-\left.\frac{d}{d\lambda}\right|_{\lambda=1^+}\int_{0}^{\infty} w\left(\lambda t\right)w\left(\frac{t}{\lambda}\right)dt
\end{equation}
on
\[w(t)=A\log^-|t-1|+B\chi_{[0,1]}(t)+h(t),\]
where $\log^-t=\min\{\log t,0\}$, $A$ and $B$ are real numbers, and $h$ is a function satisfying, for some constants $\alpha\in(0,1)$, $\gamma\in(0,1)$, and $C_0$, the following conditions:
\begin{itemize}
\item[(i)] $\|h\|_{C^{\alpha}((0,\infty))}\leq C_0$.
\item[(ii)] For all $\beta\in[\gamma,1+\gamma]$, \[\|h\|_{C^{\beta}((0,1-\rho)\cup(1+\rho,2))}\leq C_0 \rho^{-\beta}\qquad \textrm{for all}\ \ \rho\in(0,1).\]
\item[(iii)] $|h'(t)|\leq C t^{-2-\gamma}$ and $|h''(t)|\leq C t^{-3-\gamma}$ for all $t>2$.
\end{itemize}

We will split the proof of Proposition \ref{propoperador} into three parts.
The first part is the following, and evaluates the operator $\mathfrak I$ on the function
\begin{equation}\label{w0}w_0(t)=A\log^-|t-1|+B\chi_{[0,1]}(t).\end{equation}

\begin{lem}\label{c2} Let $w_0$ and $\mathfrak I$ be given by \eqref{w0} and \eqref{mathfrakI}, respectively.
Then,
\[\mathfrak I(w_0)=A^2\pi^2+B^2.\]
\end{lem}

The second result towards Proposition \ref{propoperador} is the following.

\begin{lem}\label{lema1} Let $h$ be a function satisfying {\rm(i)}, {\rm(ii)}, and {\rm(iii)} above, and $\mathfrak I$ be given by \eqref{mathfrakI}.
Then,
\[\mathfrak I(h)=0.\]
Moreover, there exist constants $C$ and $\nu>1$, depending only on the constants $\alpha$, $\gamma$, and $C_0$ appearing in {\rm(i)-(ii)-(iii)}, such that
\[\left|\int_{0}^{\infty} \left\{h\left(\lambda t\right)h\left(\frac{t}{\lambda}\right)-h(t)^2\right\}dt\right|\leq C|\lambda-1|^\nu\]
for each $\lambda\in(1,3/2)$.
\end{lem}

Finally, the third one states that $\mathfrak I(w_0+h)=\mathfrak I(w_0)$ whenever $\mathfrak I(h)=0$.

\begin{lem}\label{nodependeh} Let $w_1$ and $w_2$ be $L^2(\R)$ functions. Assume that the derivative at $\lambda=1^+$ in the expression $\mathfrak I(w_1)$ exists, and that
\[\mathfrak I(w_2)=0.\]
Then,
\[\mathfrak I(w_1+w_2)=\mathfrak I(w_1).\]
\end{lem}

Let us now give the proofs of Lemmas \ref{c2}, \ref{lema1}, and \ref{nodependeh}.
We start proving Lemma \ref{nodependeh}. For it, is useful to introduce the bilinear form
\[(w_1,w_2)=-\frac12\left.\frac{d}{d\lambda}\right|_{\lambda=1^+}\int_{0}^{\infty} \left\{w_1\left(\lambda t\right)w_2\left(\frac{t}{\lambda}\right)+w_1\left(\frac{t}{\lambda}\right)w_2\left(\lambda t\right)\right\}dt,\]
and more generally, the bilinear forms
\begin{equation}\label{pelambda}
(w_1,w_2)_\lambda=-\frac{1}{2(\lambda-1)}\int_{0}^{\infty} \left\{w_1\left(\lambda t\right)w_2\left(\frac{t}{\lambda}\right)+w_1\left(\frac{t}{\lambda}\right)w_2\left(\lambda t\right)-2w_1(t)w_2(t)\right\}dt,
\end{equation}
for $\lambda>1$.

It is clear that $\lim_{\lambda\downarrow 1}(w_1,w_2)_\lambda=(w_1,w_2)$ whenever the limit exists, and that $(w,w)=\mathfrak{I}(w)$.
The following lemma shows that these bilinear forms are positive definite and, thus, they satisfy the Cauchy-Schwarz inequality.

\begin{lem}\label{scalar} The following properties hold.
\begin{itemize}
\item[(a)] $(w_1,w_2)_\lambda$ is a bilinear map.
\item[(b)] $(w,w)_\lambda\geq0$ for all $w\in L^2(\mathbb R_+)$.
\item[(c)] $|(w_1,w_2)_\lambda|^2\leq (w_1,w_1)_\lambda (w_2,w_2)_\lambda$.
\end{itemize}
\end{lem}

\begin{proof} Part (a) is immediate. Part (b) follows from the H\"older inequality
\[\|w_{\lambda}w_{1/\lambda}\|_{L^1}\leq \|w_\lambda\|_{L^2}\|w_{1/\lambda}\|_{L^2}=\|w\|_{L^2}^2,\]
where $w_{\lambda}(t)=w(\lambda t)$.
Part (c) is a consequence of (a) and (b).
\end{proof}

Now, Lemma \ref{nodependeh} is an immediate consequence of this Cauchy-Schwarz inequality.

\begin{proof}[Proof of Lemma \ref{nodependeh}] By Lemma \ref{scalar} (iii) we have
\[0\leq |(w_1,w_2)_\lambda|\leq \sqrt{(w_1,w_1)_\lambda}\sqrt{(w_2,w_2)_\lambda}\longrightarrow0.\]
Thus, $(w_1,w_2)=\lim_{\lambda\downarrow1}(w_1,w_2)_\lambda=0$ and
\[\mathfrak I(w_1+w_2)=\mathfrak I(w_1)+\mathfrak I(w_2)+2(w_1,w_2)=\mathfrak I(w_1).\]
\end{proof}

Next we prove that $\mathfrak I(h)=0$. For this, we will need a preliminary lemma.

\begin{lem}\label{h} Let $h$ be a function satisfying {\rm(i)}, {\rm(ii)}, and {\rm(iii)} in Propostion \ref{propoperador}, $\lambda\in(1,3/2)$, and $\tau\in(0,1)$ be such that $\tau/2>\lambda-1$.
Let $\alpha$, $\gamma$, and $C_0$ be the constants appearing in {\rm(i)-(ii)-(iii)}. Then,
\[
\left|h(\lambda t)h\left(\frac{t}{\lambda}\right)-h(t)^2\right|\leq \left\{
\begin{array}{ll}\displaystyle
C\max\left\{\left|t-\lambda\right|^{\alpha},\left|t-1/\lambda\right|^{\alpha}\right\} & t\in(1-\tau,1+\tau)\vspace{2mm}\\ \displaystyle
C(\lambda-1)^{1+\gamma}|t-1|^{-1-\gamma} & t\in(0,1-\tau)\cup(1+\tau,2)\vspace{2mm}\\ \displaystyle
C(\lambda-1)^2t^{-1-\gamma} & t\in(2,\infty),
\end{array}
\right.\]
where the constant $C$ depends only on $C_0$.
\end{lem}

\begin{proof} Let $t\in(1-\tau,1+\tau)$. Let us denote $\widetilde h=h-h(1)$. Then,
\[h\left(\lambda t\right)h\left(\frac{t}{\lambda}\right)-h(t)^2=
\widetilde h\left(\lambda t\right)\widetilde h\left(\frac{t}{\lambda}\right)-\widetilde h(t)^2+h(1)\left(\widetilde h\left(\lambda t\right)+\widetilde h\left(\frac{t}{\lambda}\right)-2\widetilde h(t)\right).\]
Therefore, using that $|\widetilde h(t)|\leq C_0|t-1|^\alpha$ and $\|h\|_{L^\infty(\R)}\leq C_0$, we obtain
\begin{eqnarray*}
\left|h\left(\lambda t\right)h\left(\frac{t}{\lambda}\right)-h(t)^2\right|&\leq&
C\left|\lambda t-1\right|^{\alpha}\left|\frac{t}{\lambda}-1\right|^{\alpha}+
C|t-1|^{2\alpha}+C|\lambda t-1|^{\alpha}+\\
& & +C\left|\frac{t}{\lambda}-1\right|^{\alpha}+C|t-1|^\alpha \\
&\leq&C\max\left\{\left|t-\lambda\right|^{\alpha},\left|t-\frac{1}{\lambda}\right|^{\alpha}\right\}.
\end{eqnarray*}

Let now $t\in(0,1-\tau)\cup(1+\tau,2)$ and recall that $\lambda\in(1,1+\tau/2)$.
Define, for $\mu\in[1,\lambda]$,
\[\psi(\mu)=h\left(\mu t\right)h\left(\frac{t}{\mu}\right)-h(t)^2.\]
By the mean value theorem, $\psi(\lambda)=\psi(1)+\psi'(\mu)(\lambda-1)$ for some $\mu\in(1,\lambda)$.
Moreover, observing that $\psi(1)=\psi'(1)=0$, we deduce
\[|\psi(\lambda)|\leq (\lambda-1)|\psi'(\mu)-\psi'(1)|.\]

Next we claim that
\begin{equation}\label{psimu} |\psi'(\mu)-\psi'(1)|\leq C|\mu-1|^\gamma|t-1|^{-1-\gamma}.\end{equation}
This yields the desired bound for $t\in(0,1-\tau)\cup(1+\tau,2)$.

To prove this claim, note that
\[\psi'(\mu)=th'\left(\mu t\right)h\left(\frac{t}{\mu}\right)-\frac{t}{\mu^2}h\left(\mu t\right)h'\left(\frac{t}{\mu}\right).\]
Thus, using the bounds from (ii) with $\beta$ replaced by $\gamma$, $1$, and $1+\gamma$,
\[\begin{split}
|\psi'(\mu)-&\psi'(1)|\leq t|h'(\mu t)-h'(t)| \left|h\left(\frac{t}{\mu}\right)\right|+t\left|h\left(\frac{t}{\mu}\right)-h(t)\right| |h'(t)|+\\
&\hspace{20mm}+t\left|h'\left(\frac{t}{\mu}\right)-h'(t)\right| \frac{|h(\mu t)|}{\mu^2}+t\left|\frac{h(\mu t)}{\mu^2}-h(t)\right||h'(t)|\\
&\leq C|\mu t-t|^\gamma m^{-1-\gamma}+
C\left|\frac t\mu-t\right|^\gamma m^{-\gamma}|t-1|^{-1}+
\frac{C}{\mu^2}\left|\frac t\mu-t\right|^\gamma m^{-1-\gamma}+\\
&\hspace{10mm}+\frac{C}{\mu^2}|\mu t-t|^\gamma m^{-\gamma}|t-1|^{-1}+C(\mu-1)|t-1|^{-1}\\
&\leq  C(\mu-1)^\gamma m^{-1-\gamma},
\end{split}\]
where $m=\min\left\{|\mu t-1|,|t-1|,|t/\mu-1|\right\}$.

Furthermore, since $\mu-1<|t-1|/2$, we have $m\geq \frac14|t-1|$, and hence \eqref{psimu} follows.

Finally, if $t\in(2,\infty)$, with a similar argument but using the bound (iii) instead of (ii), we obtain
\[|\psi(\lambda)|\leq C(\lambda-1)^2t^{-1-\gamma},\]
and we are done.
\end{proof}

Let us now give the

\begin{proof}[Proof of Lemma \ref{lema1}] Let us call
\[I_\lambda=\int_{0}^{\infty}\left\{h\left(\lambda t\right)h\left(\frac{t}{\lambda}\right)-h(t)^2\right\}dx.\]
For each $\lambda\in(1,3/2)$, take $\tau\in(0,1)$ such that $\lambda-1<\tau/2$ to be chosen later. Then, by Lemma \ref{h},
\begin{eqnarray*}
|I_\lambda| &\leq &
C(\lambda-1)^{1+\gamma}\int_0^{1-\tau}|t-1|^{-1-\gamma}dt+C\int_{1-\tau}^1\left|t-\lambda\right|^{\alpha}dt+\\
& &+C\int_1^{1+\tau}\left|t-\frac{1}{\lambda}\right|^{\alpha}dt+C(\lambda-1)^{1+\gamma}\int_{1+\tau}^2 |t-1|^{-1-\gamma}dt+\\
&& + C(\lambda-1)^2\int_{2}^\infty t^{-1-s}dt\\
&\leq& C(\lambda-1)^{1+\gamma}\tau^{-\gamma}+C\left(\tau+\lambda-1\right)^{\alpha+1}+C(\lambda-1)^{1+\gamma}\tau^{-\gamma}+\\
&&+C\left(\tau+1-\frac{1}{\lambda}\right)^{\alpha+1}+C(\lambda-1)^2.
\end{eqnarray*}

Choose now
\[\tau=(\lambda-1)^{\theta},\]
with $\theta<1$ to be chosen later.
Then,
\[\tau+\lambda-1\leq 2\tau\qquad \mbox{and} \qquad \tau+1-\frac1\lambda\leq 2\tau,\]
and hence
\[\left|I_\lambda\right|\leq
C(\lambda-1)^{(\alpha+1)\theta}+C(\lambda-1)^{1+\gamma-\theta\gamma}+C(\lambda-1)^2.\]

Finally, choose $\theta$ such that $(\alpha+1)\theta>1$ and $1+\gamma-\theta\gamma>1$, that is, satisfying
\[\frac{1}{1+\alpha}<\theta<1.\]
Then, for $\nu=\min\{(\alpha+1)\theta, 1+\gamma-\gamma\theta\}>1$, it holds
\[\left|\int_{0}^{\infty}\left\{h\left(\lambda t\right)h\left(\frac{t}{\lambda}\right)-h(t)^2\right\}dt\right|\leq C|\lambda-1|^\nu,\]
as desired.
\end{proof}

Next we prove Lemma \ref{c2}.

\begin{proof}[Proof of Lemma \ref{c2}] Let
\[w_1(t)=\log^-|t-1|\qquad\textrm{and}\qquad w_2(t)=\chi_{[0,1]}(t).\]
We will compute first $\mathfrak I(w_1)$.

Define
\[\Psi(t)=\int_0^t\frac{\log|r-1|}{r}dr.\]
It is straightforward to check that, if $\lambda>1$, the function
\begin{eqnarray*} \vartheta_\lambda(t)&=&\left(t-\frac{1}{\lambda}\right)\log|\lambda t-1|\log\left|\frac{t}{\lambda}-1\right|+
(\lambda-t)\log\left|\frac{t}{\lambda}-1\right|\\
& &-\frac{\lambda^2-1}{\lambda}\log(\lambda^2-1)\log\left|\frac{t}{\lambda}-1\right|-
\frac{\lambda^2-1}{\lambda}\Psi\left(\frac{\lambda(\lambda-t)}{\lambda^2-1}\right)\\
& &+2t-\frac{\lambda t-1}{\lambda}\log|\lambda t-1|
\end{eqnarray*}
is a primitive of $\log|\lambda t-1|\log\left|\frac{t}{\lambda}-1\right|$.
Denoting $I_\lambda=\int_{0}^{\infty} w_1\left(\lambda t\right)w_1\left(\frac{t}{\lambda}\right)dt$, we have
\begin{eqnarray*}I_\lambda-I_1&=&\int_{0}^{\frac{2}{\lambda}}\log|\lambda t-1|\log\left|\frac{t}{\lambda}-1\right|dt-\int_0^2\log^2|t-1|dt\\
&=&\vartheta_\lambda\left(\frac{2}{\lambda}\right)-\vartheta_\lambda(0)-4\\
&=&\left(\frac{\lambda^2-1}{\lambda}\right)
\left\{\Psi\left(\frac{\lambda^2}{\lambda^2-1}\right)-
\Psi\left(\frac{\lambda^2-2}{\lambda^2-1}\right)\right\}+\left(\lambda-\frac{2}{\lambda}\right)\log\left(\frac{2}{\lambda^2}-1\right)+\\
&&+\left(\lambda-\frac{1}{\lambda}\right)\log(\lambda^2-1)\log\left(\frac{2}{\lambda^2}-1\right)
-\frac{4(\lambda-1)}{\lambda},
\end{eqnarray*}
where we have used that
\[I_1=\int_0^2\log^2|t-1|dt=2\int_0^1\log^2t'dt'=2\int_0^\infty r^2e^{-r}dr=2\Gamma(3)=4.\]

Therefore, dividing by $\lambda-1$ and letting $\lambda\downarrow1$,
\begin{eqnarray*}\left.\frac{d}{d\lambda}\right|_{\lambda=1^+}I_\lambda&=&
2\lim_{\lambda\downarrow 1}\int_{\frac{\lambda^2-2}{\lambda^2-1}}^{\frac{\lambda^2}{\lambda^2-1}}\frac{\log|t-1|}{t}\,dt+\\
& &+\lim_{\lambda\downarrow 1}\left\{2\log(\lambda^2-1)\log\left(\frac{2}{\lambda^2}-1\right)
-\frac{\log\left(\frac{2}{\lambda^2}-1\right)}{\lambda-1}-\frac{4}{\lambda}\right\}.\end{eqnarray*}
The first term equals to
\[\lim_{M\rightarrow +\infty}\int_{-M}^{M}\frac{2\log|t-1|}{t}dt,\]
while the second, using that $\log(1+x)\sim x$ for $x\sim 0$, equals to
\[\lim_{\lambda\downarrow 1}\left\{2\log(\lambda^2-1)\left(\frac{2}{\lambda^2}-2\right)-\frac{\frac{2}{\lambda^2}-2}{\lambda-1}-\frac{4}{\lambda}\right\}=0+4-4=0.\]
Hence,
\begin{eqnarray*}\left.\frac{d}{d\lambda}\right|_{\lambda=1^+}I_\lambda&=&\lim_{M\rightarrow +\infty}\int_{-M}^{M}\frac{2\log|t-1|}{t}dt=\lim_{M\rightarrow +\infty}\int_{-M}^{M}\frac{2\log|t|}{t+1}dt\\
&=&\lim_{M\rightarrow +\infty}\left\{\int_{-M}^{0}\frac{2\log(-t)}{t+1}dt+\int_{0}^{M}\frac{2\log t}{t+1}dt\right\}\\
&=&\lim_{M\rightarrow +\infty}\left\{\int_{0}^{M}\frac{2\log t}{1-t}dt+\int_{0}^{M}\frac{2\log t}{t+1}dt\right\}=\int_{0}^{+\infty}\frac{4\log t}{1-t^2}dt\\
&=&\int_{0}^{1}\frac{4\log t}{1-t^2}dt+\int_{1}^{+\infty}\frac{-4\log\frac{1}{t}}{\frac{1}{t^2}-1}\frac{dt}{t^2}=2\int_{0}^{1}\frac{4\log t}{1-t^2}dt.
\end{eqnarray*}
Furthermore, using that $\frac{1}{1-t^2}=\sum_{n\geq0}t^{2n}$
and that
\[\int_0^1t^n\log t\ dt=-\int_0^1\frac{t^{n+1}}{n+1}\frac{1}{t}dt=-\frac{1}{(n+1)^2},\]
we obtain
\[\int_{0}^{1}\frac{\log t}{1-t^2}dt=-\sum_{n\geq0}\frac{1}{(2n+1)^2}=-\frac{\pi^2}{8},\]
and thus
\[\mathfrak I(w_1)=-\left.\frac{d}{d\lambda}\right|_{\lambda=1^+}I_\lambda=\pi^2.\]

Let us evaluate now $\mathfrak I(w_2)=\mathfrak I(\chi_{[0,1]})$. We have
\[\int_{0}^{+\infty} \chi_{[0,1]}\left(\lambda t\right)\chi_{[0,1]}\left(\frac{t}{\lambda}\right)dt=\int_{0}^{\frac{1}{\lambda}}dt=\frac{1}{\lambda}.\]
Therefore, differentiating with respect to $\lambda$ we obtain $\mathfrak I(w_2)=1$.

Let us finally prove that $(w_1,w_2)=0$, i.e., that
\begin{equation}\label{derAB}
\left.\frac{d}{d\lambda}\right|_{\lambda=1^+}\left\{\int_0^{\lambda}\log|1-\lambda t|dt+\int_0^{\frac{1}{\lambda}}\log\left|1-\frac{t}{\lambda}\right|dt\right\}=0.
\end{equation}
We have
\begin{eqnarray*}\int_0^{\lambda}\log|1-\lambda t|dt&=&\frac{1}{\lambda}\bigl[(\lambda t-1)\log|1-\lambda t|-\lambda t\bigr]_0^\lambda\\
&=&\left(\lambda-\frac{1}{\lambda}\right)\log(\lambda^2-1)-\lambda,
\end{eqnarray*}
and similarly,
\[\int_0^{\frac{1}{\lambda}}\log\left|1-\frac{t}{\lambda}\right|dt= \left(\frac{1}{\lambda}-\lambda\right)\log\left(1-\frac{1}{\lambda^2}\right)-\frac{1}{\lambda}.
\]
Thus,
\[\left|\int_0^{\lambda}\log|1-\lambda t|dt+\int_0^{\frac{1}{\lambda}}\log\left|1-\frac{t}{\lambda}\right|dt-2\int_0^1\log|1-t|dt\right|=\]
\[=\left|\frac{2(\lambda^2-1)}{\lambda}\log\lambda-\frac{(\lambda-1)^2}{\lambda}\right|\leq 4(\lambda-1)^2.\]
Therefore \eqref{derAB} holds, and the proposition is proved.
\end{proof}

Finally, to end this section, we give the:

\begin{proof}[Proof of Proposition \ref{propoperador}] Let us write $\varphi=w_0+h$, where $w_0$ is given by \eqref{w0}. Then, for each $\lambda>1$ we have
\[(\varphi,\varphi)_\lambda=(w_0,w_0)_\lambda+2(w_0,h)_\lambda+(h,h)_\lambda,\]
where $(\cdot,\cdot)_\lambda$ is defined by \eqref{pelambda}.
Using Lemma \ref{scalar} (c) and Lemma \ref{lema1}, we deduce
\[\left|(\varphi,\varphi)_\lambda-A^2\pi^2-B^2\right|\leq \left|(w_0,w_0)_\lambda-A^2\pi^2-B^2\right|+C|\lambda-1|^\nu.\]
The constants $C$ and $\nu$ depend only on $\alpha$, $\gamma$, and $C_0$, and by Lemma \ref{c2} the right hand side goes to $0$ as $\lambda\downarrow1$, since $(w_0,w_0)_\lambda\rightarrow \mathfrak I(w_0)$ as $\lambda\downarrow1$.
\end{proof}

\section{Proof of the Pohozaev identity in non-star-shaped domains} \label{sec8}

In this section we prove Proposition \ref{intparts} for general $C^{1,1}$ domains.
The key idea is that every $C^{1,1}$ domain is locally star-shaped, in the sense that its intersection with any small ball is star-shaped with respect to some point.
To exploit this, we use a partition of unity to split the function $u$ into a set of functions $u_1$, ..., $u_m$, each one with support in a small ball.
However, note that the Pohozaev identity is quadratic in $u$, and hence we must introduce a bilinear version of this identity, namely
\begin{equation}\label{bilinear}\begin{split}\int_\Omega(x\cdot\nabla &u_1)(-\Delta)^su_2\, dx+\int_\Omega(x\cdot\nabla u_2)(-\Delta)^su_1\,dx=\frac{2s-n}{2}\int_\Omega u_1(-\Delta)^su_2\, dx+\\
&+\frac{2s-n}{2}\int_{\Omega}u_2(-\Delta)^su_1\, dx-\Gamma(1+s)^2\int_{\partial\Omega}\frac{u_1}{\delta^{s}}\frac{u_2}{\delta^{s}}(x\cdot\nu)\, d\sigma.\end{split}\end{equation}

The following lemma states that this bilinear identity holds whenever the two functions $u_1$ and $u_2$ have disjoint compact supports.
In this case, the last term in the previous identity equals 0, and since $(-\Delta)^s u_i$ is evaluated only outside the support of $u_i$, we only need to require $\nabla u_i\in L^1(\R^n)$.

\begin{lem}\label{duesboles} Let $u_1$ and $u_2$ be $W^{1,1}(\R^n)$ functions with disjoint compact supports $K_1$ and $K_2$.
Then,
\[\begin{split}\int_{K_1}(x\cdot\nabla u_1)(-\Delta)^su_2\, dx&+\int_{K_2}(x\cdot\nabla u_2)(-\Delta)^su_1\,dx=\\
&=\frac{2s-n}{2}\int_{K_1}u_1(-\Delta)^su_2\, dx+\frac{2s-n}{2}\int_{K_2}u_2(-\Delta)^su_1\, dx.\end{split}\]
\end{lem}

\begin{proof} We claim that
\begin{equation}\label{idcabre}
(-\Delta)^s(x\cdot\nabla u_i)=x\cdot\nabla(-\Delta)^su_i+2s(-\Delta)^su_i\qquad \mbox{in}\ \  \R^n\backslash K_i.
\end{equation}
Indeed, using $u_i \equiv 0$ in $\R^n\setminus K_i$ and the definition of $(-\Delta)^s$ in \eqref{laps}, for each $x\in \R^n\backslash K_i$ we have
\[\begin{split}
(-\Delta)^s(x\cdot\nabla u_i)(x)&=c_{n,s}\int_{K_i}\frac{-y\cdot\nabla u_i(y)}{|x-y|^{n+2s}}dy\\
&=c_{n,s}\int_{K_i}\frac{(x-y)\cdot\nabla u_i(y)}{|x-y|^{n+2s}}dy+c_{n,s}\int_{K_i}\frac{-x\cdot\nabla u_i(y)}{|x-y|^{n+2s}}dy\\
&=c_{n,s}\int_{K_i}{\rm div}_y\left(\frac{x-y}{|x-y|^{n+2s}}\right)u_i(y)dy+x\cdot(-\Delta)^s\nabla u_i(x)\\
&=c_{n,s}\int_{K_i}\frac{-2s}{|x-y|^{n+2s}}u_i(y)dy+x\cdot\nabla(-\Delta)^su_i(x)\\
&=2s(-\Delta)^su_i(x)+x\cdot\nabla(-\Delta)^su_i(x),
\end{split}\]
as claimed.

We also note that for all functions $w_1$ and $w_2$ in $L^1(\R^n)$ with disjoint compact supports $W_1$ and $W_2$, it holds the integration by parts formula
\begin{equation}\label{above}
\int_{W_1}w_1(-\Delta)^sw_2=\int_{W_1}\int_{W_2}\frac{-w_1(x)w_2(y)}{|x-y|^{n+2s}}dy\,dx
=\int_{W_2}w_2(-\Delta)^sw_1.
\end{equation}

Using that $(-\Delta)^s u_2$ is smooth in $K_1$ and integrating by parts,
\[\int_{K_1}(x\cdot \nabla u_1)(-\Delta)^su_2= -n\int_{K_1}u_1(-\Delta)^su_2-\int_{K_1}u_1x\cdot\nabla(-\Delta)^su_2.\]
Next we apply the previous claim and also the integration by parts formula \eqref{above} to $w_1=u_1$ and $w_2=x\cdot \nabla u_2$.
We obtain
\[\begin{split}
\int_{K_1}u_1x\cdot\nabla(-\Delta)^su_2&=\int_{K_1}u_1(-\Delta)^s(x\cdot \nabla u_2)-2s\int_{K_1}u_1(-\Delta)^su_2\\
&=\int_{K_2}(-\Delta)^su_1(x\cdot \nabla u_2)-2s\int_{K_1}u_1(-\Delta)^su_2.
\end{split}\]
Hence,
\[\int_{K_1}(x\cdot \nabla u_1)(-\Delta)^su_2=-\int_{K_2}(-\Delta)^su_1(x\cdot \nabla u_2)+(2s-n)\int_{K_1}u_1(-\Delta)^su_2.\]

Finally, again by the integration by parts formula \eqref{above} we find
\[\int_{K_1}u_1(-\Delta)^su_2=\frac12\int_{K_1}u_1(-\Delta)^su_2+\frac12\int_{K_2}u_2(-\Delta)^su_1,\]
and the lemma follows.
\end{proof}

The second lemma states that the bilinear identity \eqref{bilinear} holds whenever the two functions $u_1$ and $u_2$ have compact supports in a ball $B$ such that $\Omega\cap B$ is star-shaped with respect to some point $z_0$ in $\Omega\cap B$.

\begin{lem}\label{unabola} Let $\Omega$ be a bounded $C^{1,1}$ domain, and let $B$ be a ball in $\R^n$.
Assume that there exists $z_0\in \Omega\cap B$ such that
\[(x-z_0)\cdot\nu(x)>0\qquad\mbox{for all}\ x\in\partial\Omega\cap \overline B.\]
Let $u$ be a function satisfying the hypothesis of Proposition \ref{intparts}, and let $u_1=u\eta_1$ and $u_2=u\eta_2$, where $\eta_i\in C^\infty_c(B)$, $i=1,2$.
Then, the following identity holds
\[\int_B(x\cdot\nabla u_1)(-\Delta)^su_2\, dx+\int_B(x\cdot\nabla u_2)(-\Delta)^su_1\,dx=\frac{2s-n}{2}\int_Bu_1(-\Delta)^su_2\, dx+\]
\[+\frac{2s-n}{2}\int_{B}u_2(-\Delta)^su_1\, dx
-\Gamma(1+s)^2\int_{\partial\Omega\cap B}\frac{u_1}{\delta^{s}}\frac{u_2}{\delta^{s}}(x\cdot\nu)\, d\sigma.\]
\end{lem}

\begin{proof} We will show that given $\eta\in C^\infty_c(B)$ and letting $\tilde u=u\eta$ it holds
\begin{equation}\label{y}
\int_B(x\cdot \nabla \tilde u)(-\Delta)^{s}\tilde u\,dx=\frac{2s-n}{2}\int_B\tilde u(-\Delta)^s\tilde u\,dx-\Gamma(1+s)^2\int_{\partial\Omega\cap B}\left(\frac{\tilde u}{\delta^{s}}\right)^2(x\cdot\nu) d\sigma.
\end{equation}
From this, the lemma follows by applying \eqref{y} with $\tilde u$ replaced by $(\eta_1+\eta_2)u$ and by $(\eta_1-\eta_2)u$, and subtracting both identities.

We next prove \eqref{y}.
For it, we will apply the result for strictly star-shaped domains, already proven in Section \ref{sec2}.
Observe that there is a $C^{1,1}$ domain $\tilde \Omega$ satisfying
\[ \{\tilde u>0\}\subset \tilde\Omega\subset \Omega\cap B\quad \mbox{and}\quad(x-z_0)\cdot\nu(x)>0 \quad
\mbox{for all }x\in \partial \tilde \Omega.\]
This is because, by the assumptions, $\Omega\cap B$ is a Lipschitz polar graph about the point $z_0\in \Omega\cap B$ and
${\rm supp}\,\tilde u\subset B'\subset\subset B$ for some smaller ball $B'$; see Figure \ref{figura2}.
Hence, there is room enough to round the corner that $\Omega \cap B$ has on $\partial \Omega\cap \partial B$.

\begin{figure}
\begin{center}
\includegraphics[]{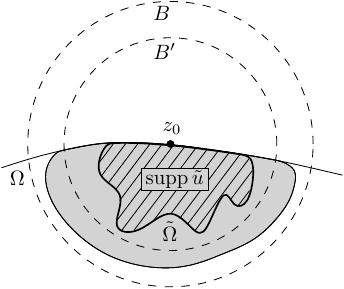}
\end{center}
\caption{\label{figura2} }
\end{figure}

Hence, it only remains to prove that $\tilde u$ satisfies the hypotheses of Proposition \ref{intparts}.
Indeed, since $u$ satisfies (a) and $\eta$ is $C^\infty_c(B')$ then $\tilde u$ satisfies
\[ [\tilde u]_{C^\beta(\{x\in \tilde \Omega\,:\, \tilde \delta(x) >\rho\})} \le C\rho^{s-\beta}\]
for all $\beta \in [s,1+2s)$, where $\tilde\delta(x)= {\rm dist}(x,\partial\tilde\Omega)$.

On the other hand, since $u$ satisfies (b) and we have $\eta \delta^s/{\tilde\delta}^s$ is Lipschitz in ${\rm supp}\, \tilde u$ ---because ${\rm dist}(x,\partial\tilde\Omega\setminus \partial\Omega)\ge c>0$ for all $x\in{\rm supp}\,\tilde u$---, then we find
\[ \bigl[\tilde u/{\tilde \delta}^s\bigr]_{C^\beta(\{x\in \tilde \Omega\,:\, \tilde \delta(x) >\rho\})} \le C\rho^{\alpha-\beta}\]
for all $\beta\in [\alpha,s+\alpha]$.

Let us see now that $\tilde u$ satisfies (c), i.e., that $(-\Delta)^s \tilde u$ is bounded. For it, we use
\[(-\Delta)^s(u\eta) = \eta(-\Delta)^s u + u(-\Delta)^s \eta - I_s(u,\eta)\]
where $I_s$ is given by \eqref{Is}, i.e.,
\[I_s(u,\eta)(x) = c_{n,s}\int_{\R^n}\frac{(u(x)-u(y))(\eta(x)-\eta(y))}{|x-y|^{n+2s}}\,dy\,.\]
The first term is bounded since $(-\Delta)^s u$ so is by hypothesis.
The second term is bounded since $\eta\in C^\infty_c(\R^n)$. The third term is bounded because $u\in C^s(\R^n)$ and $\eta\in {\rm Lip}(\R^n)$.

Therefore, $\tilde u$ satisfies the hypotheses of Proposition \ref{intparts} with $\Omega$ replaced by $\tilde\Omega$, and \eqref{y} follows taking into account that for all $x_0\in \partial \tilde\Omega\cap {\rm supp}\, \tilde u =  \partial \Omega\cap {\rm supp}\, \tilde u$ we have
\[ \lim_{x\to x_0,\, x\in\tilde\Omega} \frac{\tilde u (x)}{{\tilde \delta}^s(x)}=\lim_{x\to x_0,\, x\in\Omega} \frac{\tilde u (x)}{\delta^s(x)}.\]
\end{proof}

We now give the

\begin{proof}[Proof of Proposition \ref{intparts}] Let $B_1,...,B_m$ be balls of radius $r>0$ covering $\overline\Omega$. By regularity of the domain, if $r$ is small enough, for each
$i,j$ such that $\overline{B_i}\cap\overline{B_j}\neq\varnothing$ there exists a ball $B$ containing $B_i\cup B_j$ and a point $z_0\in \Omega\cap B$ such that
\[(x-z_0)\cdot\nu(x)>0\qquad\mbox{for all}\ x\in\partial\Omega\cap B.\]

Let $\{\psi_k\}_{k=1,...,m}$ be a partition of the unity subordinated to $B_1,...,B_m$, that is, a set of smooth functions $\psi_1,...,\psi_m$ such that $\psi_1+\cdots+\psi_m=1$ in $\Omega$ and that $\psi_k$ has compact support in $B_k$ for each $k=1,...,m$. Define $u_k=u\psi_k$.

Now, for each $i,j\in\{1,...,m\}$, if $\overline{B_i}\cap\overline{B_j}=\varnothing$ we use Lemma \ref{duesboles}, while if $\overline{B_i}\cap\overline{B_j}\neq\varnothing$ we use Lemma
\ref{unabola}. We obtain
\[\begin{split}\int_\Omega(x\cdot\nabla u_i)&(-\Delta)^su_j\, dx+\int_\Omega(x\cdot\nabla u_j)(-\Delta)^su_i\,dx=\frac{2s-n}{2}\int_{\Omega}u_i(-\Delta)^su_j\, dx+\\
&+\frac{2s-n}{2}\int_{\Omega}u_j(-\Delta)^su_i\, dx
-\Gamma(1+s)^2\int_{\partial\Omega}\frac{u_i}{\delta^{s}}\frac{u_j}{\delta^{s}}(x\cdot\nu)\, d\sigma\end{split}\]
for each $1\leq i\leq m$ and $1\leq j\leq m$. Therefore, adding these identities for $i=1,...,m$ and for $j=1,...,m$ and taking into account that $u_1+\cdots+u_m=u$, we find
\[\int_\Omega(x\cdot\nabla u)(-\Delta)^su\, dx=\frac{2s-n}{2}\int_{\Omega}u(-\Delta)^su\, dx-\frac{\Gamma(1+s)^2}{2}\int_{\partial\Omega}\left(\frac{u}{\delta^s}\right)^2(x\cdot\nu)\,d\sigma,\]
and the proposition is proved.
\end{proof}

To end this section we prove Theorem \ref{thpoh}, Proposition \ref{proppoh}, Theorem \ref{corintparts}, and Corollaries \ref{cornonexistence}, \ref{cornonexistence2}, and \ref{cornonexistence3}.

\begin{proof}[Proof of Proposition \ref{proppoh} and Theorem \ref{thpoh}] By Theorem \ref{krylov}, any solution $u$ to problem \eqref{eqlin} satisfies the hypothesis of Proposition
\ref{intparts}. Hence, using this proposition and that $(-\Delta)^su=f(x,u)$, we obtain
\[\int_{\Omega}(\nabla u\cdot x)f(x,u)dx=\frac{2s-n}{2}\int_\Omega uf(x,u)dx+\frac{\Gamma(1+s)^2}{2}\int_{\partial\Omega}\left(\frac{u}{\delta^s}\right)^2(x\cdot\nu) d\sigma.\]

On the other hand, note that $(\nabla u\cdot x)f(x,u)=\nabla \left(F(x,u)\right)\cdot x-x\cdot F_x(x,u)$. Then, integrating by parts,
\[\int_{\Omega}(\nabla u\cdot x)f(x,u)dx=-n\int_\Omega F(x,u)dx-\int_\Omega x\cdot F_x(x,u)dx.\]
If $f$ does not depend on $x$, then the last term do not appear, as in Theorem \ref{thpoh}.
\end{proof}

\begin{proof}[Proof of Theorem \ref{corintparts}] As shown in the final part of the proof of Proposition \ref{intparts} for strictly star-shaped domains given in Section \ref{sec2}, the freedom for choosing the origin
in the identity from this proposition leads to
\[\int_\Omega w_{x_i}(-\Delta)^sw\ dx=-\frac{\Gamma(1+s)^2}{2}\int_{\partial\Omega}\left(\frac{w}{\delta^s}\right)^2\nu_i\ d\sigma\]
for each $i=1,...,n$.
Then, the theorem follows by using this identity with $w=u+v$ and with $w=u-v$ and subtracting both identities.
\end{proof}

\begin{proof}[Proof of Corollaries \ref{cornonexistence}, \ref{cornonexistence2}, and \ref{cornonexistence3}] We only have to prove Corollary \ref{cornonexistence3}, since Corollaries \ref{cornonexistence} and \ref{cornonexistence2} follow immediately from it by setting $f(x,u)=f(u)$ and $f(x,u)=|u|^{p-1}u$ respectively.

By hypothesis \eqref{supercritic2}, we have
\[\frac{n-2s}{2}\int_\Omega uf(x,u)dx\geq n\int_\Omega F(x,u)dx+\int_\Omega x\cdot F_x(x,u)dx.\]
This, combined with Proposition \ref{proppoh} gives
\[\int_{\partial\Omega}\left(\frac{u}{\delta^s}\right)^2(x\cdot \nu)d\sigma\leq 0.\]
If $\Omega$ is star-shaped and inequality in \eqref{supercritic2} is strict, we obtain a contradiction.
On the other hand, if inequality in \eqref{supercritic2} is not strict but $u$ is a positive solution of \eqref{eqlin}, then by the Hopf Lemma for the fractional Laplacian (see, for instance, \cite{CRS} or Lemma 3.2 in \cite{RS}) the function $u/\delta^s$ is strictly positive in $\overline\Omega$, and we also obtain a contradiction.
\end{proof}

\appendix
\section{Calculation of the constants $c_1$ and $c_2$}

In Proposition \ref{prop:Lap-s/2-delta-s} we have obtained the following expressions for the constants $c_1$ and $c_2$:
\[c_1=c_{1,\frac s2},\qquad{\rm and}\qquad c_2=\int_0^{\infty}\left\{\frac{1-x^s}{|1-x|^{1+s}}+\frac{1+x^s}{|1+x|^{1+s}}\right\}dx,\]
where $c_{n,s}$ is the constant appearing in the singular integral expression for $(-\Delta)^{s}$ in dimension $n$.

Here we prove that the values of these constants coincide with the ones given in Proposition \ref{thlaps/2}. We start by calculating $c_1$.

\begin{prop} Let $c_{n,s}$ be the normalizing constant of $(-\Delta)^{s}$ in dimension $n$. Then,
\[c_{1,\frac s2}=\frac{\Gamma(1+s)\sin\left(\frac{\pi s}{2}\right)}{\pi}.\]
\end{prop}

\begin{proof} Recall that
\begin{equation}\label{cns}
c_{n,s}=\frac{s2^{2s}\Gamma\left(\frac{n+2s}{2}\right)}{\pi^{n/2}\Gamma(1-s)}.
\end{equation}
Thus,
\[c_{1,\frac{s}{2}}=\frac{s2^{s-1}\Gamma\left(\frac{1+s}{2}\right)}{\sqrt{\pi}\Gamma\left(1-\frac{s}{2}\right)}.\]
Now, using the properties of the Gamma function (see for example \cite{AAR})
\[\Gamma(z)\Gamma\left(z+\frac 12\right)=2^{1-2z}\sqrt\pi\Gamma(2z)\qquad \textrm{and}\qquad\Gamma(z)\Gamma(1-z)=\frac{\pi}{\sin(\pi z)},\]
we obtain
\[ c_{1,\frac{s}{2}}=\frac{s2^{s-1}}{\sqrt{\pi}}\cdot\frac{\Gamma\left(\frac{1+s}{2}\right)\Gamma\left(\frac s2\right)}{\Gamma\left(1-\frac{s}{2}\right)\Gamma\left(\frac s2\right)}=\frac{s 2^{s-1}}{\sqrt{\pi}}\cdot\frac{2^{1-s}\sqrt\pi \Gamma(s)}{\pi/\sin\left(\frac{\pi s}{2}\right)}=\frac{s\Gamma(s)\sin\left(\frac{\pi s}{2}\right)}{\pi}.\]
The result follows by using that $z\Gamma(z)=\Gamma(1+z)$.
\end{proof}

Let us now compute the constant $c_2$.

\begin{prop}\label{constc2} Let $0<s<1$. Then,
\[\int_0^\infty \left\{\frac{1-x^s}{|1-x|^{1+s}}+\frac{1+x^s}{|1+x|^{1+s}}\right\}dx=\frac{\pi}{\tan\left(\frac{\pi s}{2}\right)}.\]
\end{prop}

For it, we will need some properties of the hypergeometric function $\,_2F_1$, which we prove in the next lemma.
Recall that this function is defined as
\[\,_2F_1(a,b;c;z)=\sum_{n\geq0}\frac{(a)_n(b)_n}{(c)_n}\frac{z^n}{n!}\qquad\textrm{for }\ |z|<1,\]
where $(a)_n=a(a+1)\cdots(a+n-1)$, and by analytic continuation in the whole complex plane.

\begin{lem}\label{hypergeom} Let $\,_2F_1(a,b;c;z)$ be the ordinary hypergeometric function, and $s\in\mathbb R$. Then,
\begin{itemize}
\item[(i)] For all $z\in\mathbb C$, \[\frac{d}{dz}\left\{\frac{z^{s+1}}{s+1}\,_2F_1(1+s,1+s;2+s;z)\right\}=\frac{z^s}{(1-z)^{1+s}}.\]
\item[(ii)] If $s\in(0,1)$, then \[\lim_{x\rightarrow 1}\left\{\frac{1}{s+1}\,_2F_1(1+s,1+s;2+s;x)-\frac{1}{s(1-x)^s}\right\}=-\frac{\pi}{\sin(\pi s)}.\]
\item[(iii)] If $s\in(0,1)$, then \[\lim_{x\rightarrow +\infty}\left\{\frac{(-x)^{s+1}}{s+1}
    \,_2F_1(1\hspace{-0.5mm}+\hspace{-0.5mm}s,1\hspace{-0.5mm}+\hspace{-0.5mm}s;
    2\hspace{-0.5mm}+\hspace{-0.5mm}s;x)
    -\frac{x^{s+1}}{s+1}
    \,_2F_1(1\hspace{-0.5mm}+\hspace{-0.5mm}s,1\hspace{-0.5mm}+\hspace{-0.5mm}s;2\hspace{-0.5mm}+\hspace{-0.5mm}s;
    -x)\right\}=i \pi,\]
where the limit is taken on the real line.
\end{itemize}
\end{lem}

\begin{proof}
(i) Let us prove the equality for $|z|<1$. In this case,
\[\begin{split} \frac{d}{dz}\biggl\{\frac{z^{s+1}}{s+1}\,_2F_1&(1+s,1+s;2+s;z)\biggr\} =\frac{d}{dz}\sum_{n\geq0}\frac{(1+s)^2_n}{(2+s)_n}\frac{z^{n+1+s}}{n!(s+1)}=\\
&= \sum_{n\geq0}\frac{(1+s)_n}{n!}z^{n+s}= z^s\sum_{n\geq0}{-1-s \choose n}(-z)^n= z^s(1-z)^{-1-s},
\end{split}\]
where we have used that $(2+s)_n=\frac{n+1+s}{1+s}(1+s)_n$ and that $\frac{(a)_n}{n!}=(-1)^n{-a \choose n}$.
Thus, by analytic continuation the identity holds in $\mathbb C$.

(ii) Recall the Euler transformation (see for example \cite{AAR})
\begin{equation}\label{euler}\,_2F_1(a,b;c;x)=(1-x)^{c-a-b}\,_2F_1(c-a,c-b;c;x),\end{equation}
and the value at $x=1$
\begin{equation}\label{2F1at1}\,_2F_1(a,b;c;1)=\frac{\Gamma(c)\Gamma(c-a-b)}{\Gamma(c-a)\Gamma(c-b)}\qquad \textrm{whenever}\qquad a+b<c.\end{equation}
Hence,
\[\frac{1}{s+1}\,_2F_1(1+s,1+s;2+s;x)-\frac{1}{s(1-x)^s}=\frac{\frac{1}{s+1}\,_2F_1(1,1;2+s;x)-\frac1s}{(1-x)^s},\]
and we can use l'H\^opital's rule,
\begin{eqnarray*} \lim_{x\rightarrow 1}\frac{\frac{1}{s+1}\,_2F_1(1,1;2+s;x)-\frac1s}{(1-x)^s}&=&
\lim_{x\rightarrow 1}\frac{\frac{1}{s+1}\frac{d}{dx}\,_2F_1(1,1;2+s;x)}{-s(1-x)^{s-1}}\\
&=&-\lim_{x\rightarrow 1}\frac{(1-x)^{1-s}}{s(s+1)(s+2)}\,_2F_1(2,2;3+s;x)\\
&=&-\lim_{x\rightarrow 1}\frac{1}{s(s+1)(s+2)}\,_2F_1(1+s,1+s;3+s;x)\\
&=&-\frac{1}{s(s+1)(s+2)}\,_2F_1(1+s,1+s;3+s;1)\\
&=&-\frac{1}{s(s+1)(s+2)}\frac{\Gamma(3+s)\Gamma(1-s)}{\Gamma(2)\Gamma(2)}\\
&=&-\Gamma(s)\Gamma(1-s)\\
&=&-\frac{\pi}{\sin(\pi s)}.
\end{eqnarray*}
We have used that
\[\frac{d}{dx}\,_2F_1(1,1;2+s;x)=\frac{1}{s+2}\,_2F_1(2,2;3+s;x),\]
the Euler transformation \eqref{euler}, and the properties of the $\Gamma$ function
\[x\Gamma(x)=\Gamma(x+1),\qquad \Gamma(x)\Gamma(1-x)=\frac{\pi}{\sin(\pi x)}.\]

(iii) In \cite{B} it is proved that
\begin{equation}\label{ramanujan}\frac{\Gamma(a)\Gamma(b)}{\Gamma(a+b)}\,_2F_1(a,b;a+b;x)=
\log\frac{1}{1-x}+R+o(1)\qquad \textrm{for}\ \ x\sim 1,\end{equation}
where
\[R=-\psi(a)-\psi(b)-\gamma,\]
$\psi$ is the digamma function, and $\gamma$ is the Euler-Mascheroni constant.
Using the Pfaff transformation \cite{AAR}
\[\,_2F_1(a,b;c;x)=(1-x)^{-a}\,_2F_1\left(a,c-b;c;\frac{x}{x-1}\right)\]
and \eqref{ramanujan}, we obtain
\begin{eqnarray*} \frac{(1-x)^{1+s}}{1+s}\,_2F_1(1+s,1+s;2+s;x)
&=&\frac{1}{1+s}\,_2F_1\left(1+s,1;2+s;\frac{x}{x-1}\right)\\
&=&\log\frac{1}{1-x}+R+o(1)\ \ \ \textrm{for }\ x\sim \infty.
\end{eqnarray*}
Thus, it also holds
\[\frac{(-x)^{1+s}}{1+s}\,_2F_1(1+s,1+s;2+s;x)=\log\frac{1}{1-x}+R+o(1)\ \ \ \textrm{for }\ x\sim \infty,\]
and therefore the limit to be computed is now
\[\lim_{x\rightarrow +\infty}\left\{\left(\log\frac{1}{1-x}+R\right)-\left(\log\frac{1}{1+x}+R\right)\right\}=i \pi.\]
\end{proof}

Next we give the:

\begin{proof}[Proof of Proposition \ref{constc2}] Let us compute separately the integrals
\[I_1=\int_0^1\left\{\frac{1-x^s}{|1-x|^{1+s}}+\frac{1+x^s}{|1+x|^{1+s}}\right\}dx\]
and
\[I_2=\int_1^\infty \left\{\frac{1-x^s}{|1-x|^{1+s}}+\frac{1+x^s}{|1+x|^{1+s}}\right\}dx.\]
By Lemma \ref{hypergeom} (i), we have that
\[\int\left\{\frac{1-x^s}{(1-x)^{1+s}}+\frac{1+x^s}{(1+x)^{1+s}}\right\}dx=
\frac{1}{s}(1-x)^{-s}-\frac{x^{s+1}}{s+1}\,_2F_1(1+s,1+s;2+s;x)\]
\[-\frac{1}{s}(1+x)^{-s}+\frac{x^{s+1}}{s+1}\,_2F_1(1+s,1+s;2+s;-x).\]
Hence, using \ref{hypergeom} (ii),
\[I_1=\frac{\pi}{\sin(\pi s)}-\frac{1}{s 2^s}+\frac{1}{s+1}\,_2F_1(1+s,1+s;2+s;-1).\]
Let us evaluate now $I_2$.
As before, by Lemma \ref{hypergeom} (i),
\[\int\left\{\frac{1-x^s}{(x-1)^{1+s}}+\frac{1+x^s}{(x+1)^{1+s}}\right\}dx=
\frac{1}{s}(x-1)^{-s}+(-1)^s\frac{x^{s+1}}{s+1}\,_2F_1(1+s,1+s;2+s;x)\]
\[-\frac{1}{s}(1+x)^{-s}+\frac{x^{s+1}}{s+1}\,_2F_1(1+s,1+s;2+s;-x).\]
Hence, using \ref{hypergeom} (ii) and (iii),
\begin{eqnarray*}I_2&=&-i \pi+(-1)^s\frac{\pi}{\sin(\pi s)}+\frac{1}{s 2^s}-\frac{1}{s+1}\,_2F_1(1+s,1+s;2+s;-1)\\
&=&-i\pi+\cos(\pi s)\frac{\pi}{\sin(\pi s)}+i\sin(\pi s)\frac{\pi}{\sin(\pi s)}+\\ &&\qquad\qquad\qquad+\frac{1}{s 2^s}-\frac{1}{s+1}\,_2F_1(1+s,1+s;2+s;-1)\\
&=&\frac{\pi}{\tan(\pi s)}+\frac{1}{s 2^s}-\frac{1}{s+1}\,_2F_1(1+s,1+s;2+s;-1).
\end{eqnarray*}
Finally, adding up the expressions for $I_1$ and $I_2$, we obtain
\begin{eqnarray*}\int_0^\infty \left\{\frac{1-x^s}{|1-x|^{1+s}}+\frac{1+x^s}{|1+x|^{1+s}}\right\}dx&=&\frac{\pi}{\sin(\pi s)}+\frac{\pi}{\tan(\pi s)}=
\pi\cdot\frac{1+\cos(\pi s)}{\sin(\pi s)}\\
&=&\pi\cdot\frac{2\cos^2\left(\frac{\pi s}{2}\right)}{2\sin\left(\frac{\pi s}{2}\right)\cos\left(\frac{\pi s}{2}\right)}=\frac{\pi}{\tan\left(\frac{\pi s}{2}\right)},
\end{eqnarray*}
as desired.
\end{proof}

\begin{rem}\label{A4}
It follows from Proposition \ref{propoperador} that the constant appearing in \eqref{partdificil} (and thus in the Pohozaev identity), $\Gamma(1+s)^2$,  is given by
\[c_3=c_1^2(\pi^2+c_2^2).\]
We have obtained the value of $c_3$  by computing explicitly $c_1$ and $c_2$. However, an alternative way to obtain $c_3$ is to
exhibit an explicit solution of \eqref{eq} for some nonlinearity $f$ and apply
the Pohozaev identity to this solution.
For example, when $\Omega=B_1(0)$, the solution of
\[\left\{ \begin{array}{rcll} (-\Delta)^s u &=&1&\textrm{in }B_1(0) \\
u&=&0&\textrm{in }\mathbb R^n\backslash B_1(0)\end{array}\right.\]
can be computed explicitly \cite{G,BGR}:
\begin{equation}\label{eq:explicit-sol-ball}
u(x)=\frac{2^{-2s}\Gamma(n/2)}{\Gamma\left(\frac{n+2s}{2}\right)\Gamma(1+s)}\left(1-|x|^2\right)^s.
\end{equation}
Thus, from the identity
\begin{equation}\label{jfk}(2s-n)\int_{B_1(0)}u\ dx+2n\int_{B_1(0)}u\ dx=c_3\int_{\partial B_1(0)}\left(\frac{u}{\delta^s}\right)^2(x\cdot\nu)d\sigma\end{equation}
we can obtain the constant $c_3$, as follows.

On the one hand,
\begin{eqnarray*}\int_{B_1(0)}u\ dx&=&\frac{2^{-2s}\Gamma(n/2)}{\Gamma\left(\frac{n+2s}{2}\right)\Gamma(1+s)}\int_{B_1(0)}\left(1-|x|^2\right)^sdx\\
&=&\frac{2^{-2s}\Gamma(n/2)}{\Gamma\left(\frac{n+2s}{2}\right)\Gamma(1+s)}|S^{n-1}|\int_0^1r^{n-1}(1-r^2)^sdr\\
&=&\frac{2^{-2s}\Gamma(n/2)}{\Gamma\left(\frac{n+2s}{2}\right)\Gamma(1+s)}|S^{n-1}|\frac12\int_0^1r^{n/2-1}(1-r)^sdr\\
&=&\frac{2^{-2s}\Gamma(n/2)}{\Gamma\left(\frac{n+2s}{2}\right)\Gamma(1+s)}|S^{n-1}|\frac12\frac{\Gamma(n/2)\Gamma(1+s)}{\Gamma(n/2+1+s)},
\end{eqnarray*}
where we have used the definition of the Beta function
\[B(a,b)=\int_0^1t^{a-1}(1-t)^{b-1}dt\]
and the identity
\[B(a,b)=\frac{\Gamma(a)\Gamma(b)}{\Gamma(a+b)}.\]
On the other hand,
\begin{eqnarray*} \int_{\partial B_1(0)}\left(\frac{u}{\delta^s}\right)^2(x\cdot\nu)d\sigma
&=&\left(\frac{2^{-2s}\Gamma(n/2)}{\Gamma\left(\frac{n+2s}{2}\right)\Gamma(1+s)}\right)^2|S^{n-1}|2^{2s}.
\end{eqnarray*}

Thus, \eqref{jfk} is equivalent to
\[(n+2s)\frac{2^{-2s}\Gamma(n/2)}{\Gamma\left(\frac{n+2s}{2}\right)\Gamma(1+s)} \frac12\frac{\Gamma(n/2)\Gamma(1+s)}{\Gamma(n/2+1+s)}=
c_3\left(\frac{2^{-2s}\Gamma(n/2)}{\Gamma\left(\frac{n+2s}{2}\right)\Gamma(1+s)}\right)^22^{2s}.\]
Hence, after some simplifications,
\[c_3=\frac{\Gamma(1+s)^2}{\Gamma(n/2+1+s)}\frac{n+2s}{2}\Gamma\left(\frac{n+2s}{2}\right),\]
and using that
\[z\Gamma(z)=\Gamma(1+z)\]
one finally obtains
\[c_3=\Gamma(1+s)^2,\]
as before.
\end{rem}

\section*{Acknowledgements}

The authors thank Xavier Cabr\'e for his guidance and useful discussions on the topic of this paper.

\end{document}